\documentclass[12pt,a4paper]{amsart}

\usepackage{amssymb}
\usepackage{subfigure}
\usepackage{graphicx}
\usepackage{color}
\newtheorem{theorem}{Theorem}
\newtheorem{proposition}{Proposition}
\newtheorem{corollary}{Corollary}
\newtheorem{lemma}{Lemma}
\theoremstyle{definition}
\newtheorem{definition}{Definition}
\newtheorem{example}{Example}
\theoremstyle{remark}
\newtheorem{remark}{Remark}
\newcommand{\myu}{x}

\newcommand{\ix}{X}
\newcommand{\myI}{\mathcal{I}}
\newcommand{\norm}[1]{\Vert \hspace{0.4mm} #1 \hspace{0.4mm} \Vert}
\newcommand{\normx}[1]{\norm{#1}_X}
\newcommand{\normy}[1]{\norm{#1}_Y}
\newcommand{\myphi}{\varphi}
\newcommand{\mypsi}{\psi}
\newcommand{\mylambda}{\lambda}
\newcommand{\ualpdel}[1][\alpha]{{\myu}_{#1}^\delta}
\newcommand{\ust}{\myu^\dagger}
\newcommand{\Deltaset}{\Delta_\delta}
\newcommand{\Mdelta}{M_\delta}
\newcommand{\inset}[1]{\{ \, #1 \, \}}
\newcommand{\insetla}[1]{\Big\{ \, #1 \, \Big\}}

\newcommand{\Hdelta}{H_\delta}
\newcommand{\Hdeltab}{\widetilde{H}_\delta}
\newcommand{\Hdeltac}{\widehat{H}_\delta}
\newcommand{\alpdel}{\alpha_*}
\newcommand{\sfrac}[2]{#1/#2}
\newcommand{\alpstar}{\alpha_\ast}

\newcommand{\myb}{\varkappa}
\newcommand{\mymu}{\mu}
\newcommand{\myq}{q}
\newcommand{\myqj}{\myq^j}
\newcommand{\myqpow}{\myq^\myb}
\newcommand{\myqpowinv}{\myq^{-\myb}}
\newcommand{\myqpowinvj}{\myq^{-j\myb}}
\newcommand{\alptil}{\alpha_+}
\newcommand{\mykappa}{K_2}
\newcommand{\myd}{\varrho}
\newcommand{\cob}{c_2}
\newcommand{\coc}{c_3}
\newcommand{\cod}{c_4}
\newcommand{\coe}{c_5}
\newcommand{\cof}{c_6}
\newcommand{\cog}{c_7}
\newcommand{\coh}{c_8}

\newcommand{\Lepski}{Lepski\u\i}

\newcommand{\mybeta}{\tau_{\mbox{\tiny L}}}
\newcommand{\mygamma}{\sigma}

\newcommand{\yd}{y^\delta}
\newcommand{\xdag}{x^\dag}

\newcommand{\xad}{x_{\alpha}^\delta}
\newcommand{\xa}{x_{\alpha}}
\newcommand{\hxa}{\hat x_{\alpha}}
\newcommand{\domain}{\mathcal{D}}
\newcommand{\DF}{\mathcal{D}(F)}
\newcommand{\xbar}{\overline{x}}
\newcommand{\adelta}{\alpha(\yd,\delta)}
\newcommand{\Id}{\operatorname{I}}

\newcommand{\lr}[1]{\left( #1\right)}
\newcommand{\lrb}[1]{\Big(#1\Big)}

\newcommand{\lrc}[1]{(#1)}
\newcommand{\ra}{r_{\alpha}}
\newcommand{\myopt}{\textup{opt}}
\newcommand{\mye}[1]{e{#1}}

\title[Tikhonov regularization with oversmoothing penalty]{Nonlinear Tikhonov regularization in Hilbert
  scales with oversmoothing penalty: inspecting balancing principles }

\author{B. Hofmann}
\address{Faculty of Mathematics, Chemnitz University of Technology,
  09107 Chemnitz, Germany}
\author{C. Hofmann}
\address{Faculty of Mathematics, Chemnitz University of Technology,
  09107 Chemnitz, Germany}
\author{P. Math\'e}
\address{Weierstra{\ss} Institute for Applied Analysis and
  Stochastics, Mohrenstra\ss{}e 39, 10117 Berlin, Germany}
\author{R. Plato}
\address{Department Mathematik, Emmy-Noether-Campus, Universit\"at
  Siegen, Walter-Flex-Str.~3, 57068 Siegen, Germany}
\date{\today}

\begin{document}
\footnote{
    Dedicated to Zuhair Nashed, 
 our esteemed colleague and outstanding
    Professor, doyen of operator
    theory and regularization theory for inverse and ill-posed
    problems}
\begin{abstract}
  The analysis of Tikhonov regularization for nonlinear ill-posed equations with 
 smoothness promoting penalties is an important topic in inverse
 problem theory.
With focus on  Hilbert scale models, the case of oversmoothing
penalties, i.e.,\ when the penalty takes an infinite value at the true solution
 gained increasing interest.
The considered nonlinearity structure is as in the
 study
 B.~Hofmann and P. Math\'{e}. {\it Tikhonov regularization with oversmoothing
penalty for non-linear ill-posed problems in Hilbert scales}. Inverse Problems,
2018.
 Such analysis can address two fundamental
 questions. When is it possible to achieve order optimal
 reconstruction? How to select the regularization parameter?
 The present study complements previous ones by two main facets. First,
 an error decomposition into a smoothness dependent and a (smoothness
 independent) noise
 propagation term is derived, covering a large range of smoothness
 conditions.
 Secondly, parameter selection by balancing principles is presented. A
 detailed discussion, covering some history and variations of the
 parameter choice by balancing shows under which conditions such
 balancing principles yield order optimal reconstruction. A numerical
 case study, based on some exponential growth model, provides additional insights.
\end{abstract}

\maketitle

\section{Introduction}\label{intro}

In the past years, a new facet has found interest in the theory of
inverse problems. When considering
variational (Tikhonov-type) regularization for the stable approximate
solution of {\sl ill-posed} operator equations
\begin{equation} \label{eq:opeq} F(x)=y
\end{equation}
in Hilbert spaces the treatment
of {\sl oversmoothing penalties} gained attention. The current state of the art
concerning that facet shows gaps.
The present study  aims at closing
some of them. Specifically, we analyze variants of the
balancing principle as parameter choice for the
regularization parameter.  Moreover, we would like to illustrate the
theory by numerical case studies.

The forward operator $F: \domain(F) \subseteq X \to Y$, which is
preferably assumed to be {\sl nonlinear}, maps between the separable
infinite dimensional Hilbert spaces $X$ and $Y$ and possesses the
convex and closed subset~$\domain(F)$ of $X$ as domain of
definition.
In the sequel, let $\xdag \in \domain(F)$ be a
solution to equation (\ref{eq:opeq}) for given exact right-hand side
$y=F(\xdag) \in Y$. We throughout assume that $ \xdag \in {\rm int}(\DF)$, which means that the solution belongs to the interior of the domain $ \DF $ of the operator $F$.
Given the  noise level $\delta \ge 0$, we consider the deterministic noise
model
\begin{equation*}\|y-\yd\|_{Y} \le \delta,
\end{equation*}
which means that instead of $y$, only perturbed data
$\yd =y+\delta\,\xi\in Y$ with $\|\xi\|_Y~\le~1$ are available.

The equation (\ref{eq:opeq}) is
ill-posed at least locally at $\xdag$, and finding stable solutions requires some
regularization. We apply Tikhonov regularization with quadratic misfit
and penalty functionals in a {\sl Hilbert scale} setting.

\subsection{Hilbert scales with respect to an unbounded operator}
\label{sec:scales}

The Hilbert scale is generated by some
densely defined, unbounded and self-adjoint linear operator
$B\colon \domain(B) \subset X \to X$ with domain~$\domain(B)$. This
operator $B$ is assumed to be strictly
positive such that we have for some $\underline m>0$
\begin{equation*} \|Bx\|_{X} \geq \underline m \|x\|_{X} \quad
  \mbox{for all} \quad x\in \domain(B).
\end{equation*}
The Hilbert scale $\{X_\tau\}_{\tau \in \mathbb{R}}$,
generated by $B$, is characterized by the formulas
$X_\tau=\domain(B^\tau)$ for $\tau >0$, $X_\tau=X$ for $\tau \le 0$
and $\|x\|_\tau:=\|B^\tau x\|_X$ for $\tau \in \mathbb{R}$.  We do not
need in our setting the topological completion of the spaces
$X_\tau = X$, for $ \tau < 0 $, with respect to the norm~$ \|\cdot\|_\tau $.

\subsection{Tikhonov regularization with smoothness promoting pe\-nalty}
\label{sec:tikhonov}

The operator~$B$ is used for Tikhonov regularization in the
corresponding functional
\begin{equation}
  \label{eq:tikhonov}
  T^\delta_{\alpha}(x) := \|F(x) - \yd\|_{Y}^2 + \alpha\,\|B(x-\xbar)\|_{X}^2,
  \quad x\in\domain:=\domain(F) \cap \domain(B),
\end{equation}
where
$\|F(x) - \yd\|^2$ characterizes the misfit or fidelity term, and
$\xbar \in X$ is an initial guess occurring in the penalty functional
$\|B(x-\xbar)\|_{X}^2$. Throughout this paper we suppose that
$\xbar \in \domain$.
Given a regularization parameter~$\alpha>0$ the corresponding
regularized solutions $\xad$ to $\xdag$ are obtained as the minimizers of the
Tikhonov functional $T^\delta_\alpha$ on the set~$\domain$.  By
definition of the Hilbert scale we have that~$\|B(x-\xbar)\|_{X}^2=\|x-\xbar\|_1^2$ and
consequently $\xad \in \domain(B)=X_1$ for all data~$\yd \in Y$ and
$\alpha>0$. In order to ensure {\sl existence} and {\sl stability} of
the regularized solutions $\xad$ for all $\alpha>0$
(cf.~\cite[\S~3]{HKPS07}, \cite[Section~3.2]{Scherzetal09} and \cite[Section~4.1.1]{SKHK12}), we additionally suppose that the
forward operator $F$ is weakly sequentially continuous.

Our focus will be on oversmoothing penalties,
when~$\xdag\not\in\domain(B)$ and hence $T^\delta_{\alpha}(\xdag)=
\infty$. In this case,  the regularizing property
$T^\delta_{\alpha}(\xad) \le T^\delta_{\alpha}(\xdag)$, which often is
a basic tool
for obtaining error estimates does not yield consequences.

\subsection{State of the art}
\label{sec:state}

The seminal study for Tikhonov regularization, including the case of
oversmoothing penalty, and for  linear ill-posed
operator equations, was published by Natterer in 1984,
cf.~\cite{Natterer84}.
For linear bounded operator $F=A: X \to Y$, Natterer used for this study a two-sided inequality chain
$$
c_a \|x\|_{-a} \le \|Ax\|_Y \le C_a \|x\|_{-a}  \quad \mbox{for all}
\quad x \in X,
$$
with constants $0<c_a \le C_a < \infty$ and a {\sl degree of ill-posedness} $a>0$.
Here, we adapt this to the nonlinear mapping~$F:\domain(F) \subset X \to Y$ as
\begin{equation} \label{eq:twosided} c_{a}\,\|x - \xdag\|_{-a} \leq
  \|F(x) - F(\xdag)\|_{Y} \leq C_{a} \|x - \xdag\|_{-a} \qquad \mbox{for
    all} \;\;\, x\in \mathcal{D},
\end{equation}
and rely upon~(\ref{eq:twosided}) as
an intrinsic \emph{nonlinearity condition} for the mapping~$F$ under
consideration.

This specific nonlinearity condition was first used within the
present context in \cite{HofMat18}. It was shown that a discrepancy
principle yields optimal order convergence under the power type
smoothness assumption that~$\xdag \in X_{p},\ 0< p < 1$. One major
tool was to use certain \emph{auxiliary elements}~$\hxa$ of proximal type, which are
minimizers of the functional
\begin{equation}
  \label{eq:aux-functional}
  \widehat T_{\alpha}(x) :=
  \|x-\xdag\|_{-a}^2+\alpha\|x-\xbar\|_1^2,\quad  x\in X,
\end{equation}
and hence  belong to $X_1=\domain(B)$. This study was complemented
in~\cite{GHH20} within the proceedings {\it Inverse Problems and Related
Topics: Shanghai, China, October 12–14, 2018}, Springer, 2020,  by case studies showing intrinsic problems when using
oversmoothing penalties. The same issue contains results on  \emph{a
  priori} parameter choice~$\alpha_{\ast}\sim \delta^{(2a+2)/(a+p)}$, i.e.,\ when the smoothness~$p$ is assumed to
be known, see~\cite{HofMat20}.

However, the first decomposition of the error into a
smoothness dependent increasing term (as a function of~$\alpha$ tending to zero as $\alpha \to +0$),
and a smoothness independent decreasing term proportional
to~$\frac{\delta}{\alpha^{a/(2a+2)}}$ was developed in~\cite{HofPla20}.
As a consequence, there is something special about this study that 
norm convergence of regularized solutions to the exact solution can be shown
for a wide region of a priori parameter choices and for a specific discrepancy principle
without to make any specific smoothness assumption on $\xdag$.
It is highlighted there that such error decomposition also allows for low order convergence rates
under low order smoothness assumptions on $\xdag$. But it was observed in~\cite{HofHof20} that the error
decomposition from~\cite{HofPla20} extends to power type
smoothness~$\xdag\in X_{p},\ 0< p < 1$, and hence yields optimal
rates of convergence under the a priori parameter
choice.

\subsection{Goal of the present study}
\label{sec:goal}

The present paper complements the series of articles mentioned before.

On the one hand side, it extends the error decomposition
from~\cite{HofPla20,HofHof20} to more general smoothness
assumptions. On the other hand, we present new results to the \emph{balancing principle} for choosing the regularization parameter in Tikhonov regularization for nonlinear problems with oversmoothing penalties in a Hilbert scale setting.
This work was essentially motivated by the recent
paper~\cite{Pricop19}. Pricop-Jeckstadt has also analyzed the
balancing principle for nonlinear problems in Hilbert scales, but only
for non-oversmoothing penalties. In this sense, we try to close a gap in the theory.

The material is organized as follows.
In Section~\ref{TRO-props} we highlight the decomposition of the
  error under smoothness in terms of source conditions. This provides
  us with the required structure in order to found the balancing
  principles in Section~\ref{sec:lepski}.  We give a brief account of
  the history of such principles, and formulate several specifications
  for the setup under consideration. The results presented in this part are general and may be of independent interest.
  Finally, in Section~\ref{example} we
  discuss the exponential growth model, both theoretically, and as
  subject for a numerical case study.

\section{General error estimate for Tikhonov regularization in
  Hilbert scales with oversmoothing penalty}\label{TRO-props}
The basis for an analytical treatment of the balancing principle is formed by general error estimates for Tikhonov regularization in Hilbert scales with oversmoothing penalty.
With the inequality chain~(\ref{eq:twosided}) and for $\xdag \in {\rm int}(\DF)$ such estimates have been developed recently in \cite{HofMat18},
\cite{HofPla20}, and~\cite{HofHof20}   by using \emph{auxiliary
  elements} as minimizers of~(\ref{eq:aux-functional}). Introducing the injective linear
operator~$ G: = B^{-(2a+2)}$, the
corresponding minimizers~$\hxa$ can be expressed explicitly as
$$\hxa=\xbar+G(G+\alpha \Id)^{-1}(\xdag-\xbar)=\xdag-\alpha(G+\alpha \Id)^{-1}(\xdag-\xbar). $$

\subsection{Smoothness in terms of source conditions}
\label{sec:smoothness}

General error estimates were obtained  under general
source conditions, given in terms of index functions~$\mypsi$.
\begin{definition} 
A continuous and non-decreasing function $\varphi: (0,\infty) \to (0,\infty)$ with $\lim_{t \to +0} \varphi(t)=0$ is called {\sl index function}. We call this index function {\sl sub-linear} if there is some $t_0>0$
such that the quotient~$t/\varphi(t)$ is non-decreasing for $0<t \le t_0$.
\end{definition}
In these terms a general source condition for the unknown solution~$\xdag$ is given in the form of
\begin{equation} \label{eq:losc}
\xdag-\xbar = \mypsi(G)\, w,
\quad w \in X.
\end{equation}
Here the linear operator~$\mypsi(G)$ is obtained from the operator~$G$
by spectral calculus.

\subsection{Error decomposition}
\label{sec:error}
The balancing principle relies on an error bound in a specific form, and
the corresponding fundamental error bound is given next.
\begin{theorem} \label{thm:general}
Let $\xdag \in {\rm int}(\DF)$ and let hold the inequality chain (\ref{eq:twosided}) for given the degree of ill-posedness~$a>0$. Moreover, let $\mypsi$ be an index
function such that $\mypsi^{2a+2}$ is sub-linear. If $\xdag$
satisfies a source condition~(\ref{eq:losc}) for that $\psi$, then we have
for some constant $ c_1 > 0 $ depending on $ w $ the general error estimate
\begin{align}
\label{eq:tikh-overesti-error}
\norm{\xad -\xdag}_{X} \le c_1 \mypsi(\alpha) + \frac{\delta}{\mylambda(\alpha)},
\end{align}
where $ \mylambda(\alpha) =  \frac{1}{\mykappa}\alpha^{a/(2a+2)} $, and $ \mykappa = \max\inset{1,\frac{2}{c_a}} $.
\end{theorem}
The proof follows the arguments of Proposition~3.4 in \cite{HofPla20},
and we briefly sketch these.

\begin{proof}
A substantial ingredient of the proof is
the fact that due to \cite{MatPer03} sub-linear index functions are
qualifications for the classical Tikhonov regularization approach with
norm square penalty, and the verification of formula~(25) in
\cite{HofHof20}, which in turn is based on the bounds~(21)--(23) ibid.
As can be seen from there it is enough to show that under the above assumptions, and
for~$0 \leq \theta \leq \frac{2a +1}{2a +2}$,  we have that
\begin{equation}
  \label{eq:basic}
  \alpha^{1 - \theta} \norm{G^{\theta} \lr{G + \alpha \Id}^{-1}(\xdag -
    \xbar)}_{X}\leq C \mypsi(\alpha),\quad \alpha >0.
 \end{equation}
  To this end we start from the observation that under the source condition~(\ref{eq:losc})
  we have
  $$
\norm{G^{\theta} \lr{G + \alpha \Id}^{-1}(\xdag -
    \xbar)}_{X}\leq \norm{w}_{X}\norm{G^{\theta} \lr{G + \alpha
      \Id}^{-1}\mypsi(G)}_{X \to X}.
  $$
By introducing the residual function for (classical) Tikhonov
regularization~$\ra(t) := \alpha/(t+\alpha),\ \text{for} \ t,\alpha>0$, it
suffices to   bound
  $$ \frac 1 \alpha\norm{\ra(G) G^{\theta}\mypsi(G)}_{X \to
    X}.
  $$
 The function~$t \mapsto t^{\theta}\mypsi(t)$ plainly constitutes
 an index function. We shall establish that it is sub-linear. To
 this end we write
 $$
 \left[\frac{t}{t^{\theta}\mypsi(t)}\right]^{2a+2} = \frac{t^{(1-\theta)(2a+2)}}{\mypsi^{2a+2}(t)}
 = t^{(1-\theta)(2a+2)-1}\frac{t}{\mypsi^{2a+2}(t)}, \quad 0 < t \leq t_{0}.
 $$
 Under the made sub-linearity assumption we
 find that~$t \mapsto t^{\theta}\mypsi(t)$ is sub-linear
 provided that~$(1-\theta)(2a+2)-1 \geq 0$, which corresponds
 to~$\theta \leq \frac{2a +1}{2a +2}$. Thus in this range the
 function~$t^{\theta}\mypsi(t)$ is a qualification for Tikhonov
 regularization, and we conclude that
  $$
\frac 1 \alpha\norm{\ra(G) G^{\theta}\mypsi(G)}_{X \to
    X} \leq C \frac 1 \alpha\alpha^{\theta}\mypsi(\alpha) =
  C \alpha^{\theta-1}\mypsi(\alpha),\quad \alpha>0,
  $$
  which yields~(\ref{eq:basic}).
\end{proof}

We note that the error estimate \eqref{eq:tikh-overesti-error} does not correspond to the natural estimate
$ \normx{\xad -\xdag} \le
\normx{\xa -\xdag} + \normx{\xad - \xa } $.

We highlight the above result for the prototypical examples, as
studied previously.
\begin{example}[power-type smoothness]\label{xmpl:power}
Theorem~\ref{thm:general} applies for {\sl H\"older-type source conditions} of the form
\begin{equation} \label{eq:typeH1}
 \xdag-\xbar = B^{-p} w= G^{p/(2a+2)} \quad \mbox{for} \quad 0 < p
\le 1,
\end{equation}
 which characterize for this type the case of oversmoothing
penalties. Indeed, the corresponding function~$\psi$ is~$\psi(t)=
t^{\frac{p}{2a+2}}$, such that~$\psi^{2a+2}(t) = t^{p}$ is sub-linear
whenever~$0<p\leq 1$. Then the a priori parameter choice~$\alpha_{\ast}:=
\delta^{\frac{2a+2}{a+p}}$ yields under the source condition (\ref{eq:typeH1}) a H\"older-type convergence rate of the form
\begin{equation} \label{eq:typeH2}
\|x^\delta_{\alpha_{\ast}}-\xdag\|_X=\mathcal O\lrc{\delta^{\frac p{a+p}}} \quad \mbox{as} \quad \delta \to 0.
\end{equation}
Such power-type rates correspond to \emph{moderately ill-posed problems}.

We note that the same convergence rate (\ref{eq:typeH2}) can also be obtained for~$ \alpha_{\ast} $, determined by the discrepancy principle in the
sense of formula~(\ref{eq:intr_DP}) below, cf.~\cite{HofMat18}. 

\end{example}

\begin{example}
  [low order smoothness]\label{xmpl:log}
Theorem~\ref{thm:general}
also applies for {\sl low order source conditions} with
$\mypsi$, for  which $\mypsi^{2a+2}$ is always sub-linear. Most
prominently, we assume that there is some exponent~$\mymu>0$, such that~$
\mypsi(t) =K\, (\log^{-\mymu} (1/t))$ for $0<t \le t_0=e^{-1}$, and continuously extended as
constant for $t_0<t \le \|G\|_{X \to X}$.
Moreover,  the a priori parameter choice resulting in~$\alpha_{\ast}:=
\delta$ yields  the rate of convergence
\begin{equation} \label{eq:typeL}
\|x^\delta_{\alpha_{\ast}}-\xdag\|_X= \mathcal O\lrc{\log^{-\mymu}(1/\delta)}  \quad \mbox{as} \quad \delta \to 0.
\end{equation}
Such logarithmic rates correspond to \emph{severely ill-posed
  problems}  (cf.~the study~\cite{Hohage00}).
\end{example}
%
%
\begin{example}
[no explicit smoothness]\label{xmpl:no-smoothness}
Theorem~\ref{thm:general} provides us with an upper bound of the
error, once a smoothness condition of the form \eqref{eq:losc} is available such that
$ \psi^{2a+2} $ is sub-linear.
Now we argue as follows. For the operator~$B^{-1}$ there is an index
function~$\myphi$ such that~$\xdag-\xbar = \myphi(G)v,\
v \in X $, see~\cite{MathHof2008} for compact~$B^{-1}$,
and~\cite{HMvW2009} for the general case. In the compact case it is
shown in~\cite[Cor.~2]{MathHof2008} that this index function may be
chosen concave, and hence sub-linear.
 Then
letting~$\psi:= \varphi^{1/(2a+2)}$
we can
apply Theorem~\ref{thm:general} for this index function~$\psi$.
In
particular we conclude the following. If~$\alpdel:= \alpdel(\yd,\delta)$ is any
parameter choice such that~$\alpdel\to 0$, and also~$\delta/
\alpdel^{a/(2a+2)}\to 0$ as~$\delta\to 0$,
then~$\normx{x_{\alpdel}^{\delta}-\xdag}\to 0$. For the general case
of $ B $, which includes $ B^{-1} $ non-compact, the latter result may
be found in \cite[Thm.~4.1]{HofPla20} by an alternative proof based on the Banach--Steinhaus theorem.
\end{example}
%
%
%
%
\section{Balancing principles}
\label{sec:lepski}
Vast majority of regularization theory for ill-posed equations is
concerned with asymptotic properties of regularization, as this is
convergence, and if so, rates of convergence.
For given operator $ F: \DF \subseteq X \to Y $ between Hilbert spaces $ X $ and $ Y $,
let $ Y \ni \yd \mapsto \xad := R_{\alpha}(\yd) \in X $ be any regularization
scheme for the stable approximate determination of $ \xdag \in \DF $ from data $ \yd \in Y $ such that \linebreak
$ \normy{F(\xdag)-\yd} \le \delta $. Its error at~$\xdag$ is then considered uniformly for
admissible data as
\begin{align*}
  e(\xdag,R_{\alpha},\delta) := \sup_{\yd:\,\,\normy{F(\xdag)-\yd}\leq
    \delta}\normx{\xad-\xdag}.
\end{align*}
We call a parameter choice~$\alpha = \alpha(\yd,\delta)$
\emph{convergent} if~$e(\xdag,R_{\adelta},\delta) \allowbreak \to 0$
whenever~$\delta\to 0$. In most cases, convergence of regularization parameter choices is analyzed
\emph{uniformly on some class~$\mathcal M \subset X $}, i.e.,\ it is studied
when we have $\sup_{\xdag\in\mathcal M}e(\xdag,R_{\adelta},\delta) \to 0$.

In contrast, there are studies which highlight a different paradigm:
What is the quality of a particular regularization and parameter
choice at any given data~$\yd$? This is relevant, since in practice we
are given just one instance of data~$\yd$. Then convergence is not an
issue, rather the aim is to do the best possible for any such
instance.  Balancing principles may be used for this purpose.

\Lepski's balancing principle is most prominent for the
latter paradigm. It arose in a series of papers, starting
from~\cite{MR1091202}, and it gained special attention in statistics
within the subject of `oracle inequalities' for the purpose of model
selection since then.

Within classical regularization theory it was first
studied in~\cite{MatPer03}; fundamental discussions are given
in~\cite{PerSchock05,Mathe06}.
A variation of this statistically motivated approach was followed
starting from~\cite{Raus_Haemarik[07]}, where the above mentioned paradigm was called
quasi-optimality. The original presentation of this  idea dates back
to~\cite{Leonov91}, called point-wise pseudo-optimal. It is shown
ibid.~Theorem~7, and for
Tikhonov regularization, that the choice~$\alpha_{\ast}$, obtained as
solutions of the extremal problem
\begin{equation*}
  \norm{\alpha \frac{d \xad}{d\alpha}}_X +
  \frac{\delta}{\sqrt{\alpha}} \longrightarrow \min
\end{equation*}
is point-wise pseudo-optimal.
Here we follow the approach from~\cite{Raus_Haemarik[07],
  Haemarik_Raus[09]} by using the concept of \emph{quasi-optimality}.
\subsection{Quasi-optimality}
\label{quasiopt}
In what follows, we assume that for some fixed searched-for element
$\ust \in \ix $ and
$ \ualpdel =  R_{\alpha}(\yd) \in \ix \ (\alpha > 0, \ 0 < \delta \le \delta_0) $
obtained from some noisy data $ \yd \in Y $ and
by using some regularization scheme $ R_{\alpha} $ not further specified, the estimate
\begin{align}
\label{eq:esti-error-general}
  \norm{\ualpdel -\ust}_{X} \le \myphi(\alpha) + \frac{\delta}{\mylambda(\alpha)}
\qquad (\alpha > 0)
\end{align}
holds, where $\myphi,\ \mylambda: (0,\infty) \to (0,\infty) $ are both index
functions. The function~$\alpha \mapsto \mylambda(\alpha)$ is assumed
to be known. This approach is generic and we shall not make use of any
specific properties of the operator~$ F $, nor of any specific conditions on the noisy data $ y^\delta $.

\begin{definition}
\label{th:quasioptimal}
Suppose that the error bound~\eqref{eq:esti-error-general} holds true.
A parameter choice strategy $ \alpdel = \alpdel(\yd,\delta) \ (0 < \delta \le \delta_0) $ is
called \emph{quasi-optimal}, if there is a constant~$\cob>0$ such that
for~$0 < \delta \leq \delta_{0}$
an estimate of the following kind is satisfied:
\begin{align}
\label{eq:quasioptimal}
\normx{\ualpdel[\alpdel] -\ust} \le
\cob \inf_{\alpha > 0} \lrb{
\myphi(\alpha) + \frac{\delta}{\mylambda(\alpha)}
}.
\end{align}
%
\end{definition}
%
Note that the constant $ \cob $ in Definition~\ref{th:quasioptimal},
called the~\emph{error constant} below,
may depend on the function $ \myphi $. In addition, note that Definition~\ref{th:quasioptimal} also includes a posteriori parameter choices since in applications~$ \ualpdel $ rely on data $ y^\delta \in Y $.

We highlight a general feature of quasi-optimal parameter choice.
\begin{proposition}
  \label{th:rate-guarantee}
  Suppose that the error bound~(\ref{eq:esti-error-general}) holds true.
\begin{enumerate}
\item
If for some parameter choice
rule~$\alptil=\alptil(\yd,\delta)$ we can guarantee a rate
\begin{align*}
  \varphi(\alptil) + \frac{\delta}{\mylambda(\alptil)} =
    \mathcal O(\varrho(\delta)),\quad \text{as} \quad  \delta\to 0,
\end{align*}
  for some index function~$\varrho$,
  then any quasi-optimal parameter choice
	$ \alpdel=\alpdel(\yd,\delta) $ yields the convergence rate
$$ \normx{\ualpdel[\alpdel] -\ust} = \mathcal O(\varrho(\delta)) \quad \mbox{as} \quad \delta \to 0. $$
   \item
Any quasi-optimal parameter choice $ \alpdel=\alpdel(\yd,\delta) $ yields convergence
$ \normx{\ualpdel[\alpdel] -\ust} \to 0\, $ as $\, \delta \to 0 $.
\end{enumerate}
\end{proposition}

\begin{proof}
The first part is obvious. For the second part, one may consider the right-hand side of \eqref{eq:quasioptimal} with any parameter choice $ \alptil = \alptil(\yd,\delta) $ such that
both $ \alptil \to 0 $ and~$ \delta/\lambda(\alptil) \to 0 $ as $ \delta \to 0 $.
\end{proof}

We stress once again, the focus of quasi-optimality of a parameter choice strategy is not
convergence, rather it emphasizes an \emph{oracle property}: If (by
some oracle) we are given a parameter choice
rule~$\alptil:= \alptil(\yd,\delta)$ which realizes
\begin{equation}
  \label{eq:oracle}
\myphi(\alptil) + \frac{\delta}{\mylambda(\alptil)}
\le C  \inf_{\alpha > 0} \Big(
\myphi(\alpha) + \frac{\delta}{\mylambda(\alpha)}
\Big),
\end{equation}
then the quasi-optimal rule is (up to the constant~$\cob$) as good as
the oracle choice.

In the sub-sequent sections, we shall describe variants of the
balancing principle and we shall show that these are quasi-optimal.

\subsection{The balancing principles: setup and formulation}
\label{sec:setup-balancing}

We constrain to the following setup. First, we shall assume that the noise amplification term is of the form
\begin{align}
\label{eq:lambda-def}
\mylambda(\alpha) =  \frac{\alpha^{\myb}}{\mykappa},
\end{align}
where $ \mykappa > 0 $ and $ \myb > 0 $.
For standard regularization schemes for selfadjoint and non-selfadjoint linear problems in Hilbert spaces, we have~$ \myb = 1 $ and $ \myb = \tfrac{1}{2} $, respectively.
In the situation of Section \ref{TRO-props}, representation \eqref{eq:lambda-def} holds for $ \myb = \frac{a}{2a+2} $
and  $ \mykappa = \max\inset{1,\frac{2}{c_a}} $.
\medskip

The following result will be utilized at several occasions.
\begin{lemma}
\label{th:oracle-2}
Let
$ \alptil = \alptil(\yd,\delta) > 0 $ be any parameter choice satisfying
\begin{align}
\label{eq:oracle-2}
\myphi(\alptil) \le \coc  \frac{\delta}{\mylambda(\alptil)},
\qquad
\frac{\delta}{\mylambda(\myd \alptil)} \le \cod \myphi(\myd\alptil),
\end{align}
where $ \coc, \, \cod > 0 $ and $ \myd \ge 1 $  denote some finite constants chosen       independently of $ \delta $. Then $ \alptil $ satisfies the oracle estimate
\eqref{eq:oracle}, with a constant that may by chosen as $ C = \myd^\myb(1 +              \max\{\coc,\cod\}) $.
\end{lemma}
%

\begin{proof}
Consider the case $ 0 < \alpha \le \myd\alptil $ first. From
$ \myphi(\alptil) \le \coc \frac{\delta}{\mylambda(\alptil)} $
we then obtain
\begin{align*}
\myphi(\alptil) + \frac{\delta}{\mylambda(\alptil)}
& \le
\lrc{\coc + 1}  \frac{\delta}{\mylambda(\alptil)}
\leq \myd^\myb \lrc{\coc + 1 } \frac{\delta}{\mylambda(\alpha)}
\\
& \le \myd^\myb \lrc{\coc + 1 }
\lrb{\myphi(\alpha) + \frac{\delta}{\mylambda(\alpha)}}.
\end{align*}
Next we consider the case $ \alpha \ge \myd\alptil $. Then the estimate
$ \frac{\delta}{\mylambda(\alptil)} \le
\cod \myd^\myb \myphi(\myd\alptil) $ yields
\begin{align*}
\myphi(\alptil) + \frac{\delta}{\mylambda(\alptil)}
& \le
(1+ \cod \myd^\myb) \myphi(\myd\alptil)
\le (1+ \cod \myd^\myb)
\lrb{\myphi(\alpha) + \frac{\delta}{\mylambda(\alpha)}}.
\end{align*}
This completes the proof of the lemma.
\end{proof}

For the numerical realization of the balancing principle considered below, we utilize the following finite set of regularization parameters:
\begin{align}
\label{eq:deltaset}
\Deltaset & = \inset{\alpha_0 < \alpha_1 < \cdots < \alpha_N},
\end{align}
where each element of $ \Deltaset $ as well as $ N \ge 0 $ may depend on the noise level $ \delta $.
We further assume that the elements of $ \Deltaset $ form a finite geometric sequence, i.e.,
\begin{align}
\label{eq:geometric_sequence}
\alpha_{j} = \myqj \alpha_0, \quad j = 0,1,\ldots, N, \quad \textup{with } \myq > 1,
\end{align}
where the spacing parameter $ \myq $ is assumed to be independent of $ \delta $.
We confine the search for the regularization parameter to the set
$\Deltaset $. Given a tuning parameter~$ 0 < \mygamma \le 1 $, we consider the set
\begin{align}
\label{eq:mdelta}
\Mdelta := \insetla{ \alpha \in \Deltaset : \myphi(\alpha) \le \mygamma \frac{\delta}{\mylambda(\alpha)} }.
\end{align}
The case $ \mygamma > 1 $ does not provide any improvement and thus is excluded from the considerations below.
For $ \alpha_0 $ chosen sufficiently small and $ \alpha_N $ sufficiently large in a way such that the set
$\Mdelta$ is not empty and in addition satisfies
$\Mdelta \neq \Deltaset $,
then the maximum value~$ \alptil =
\max \Mdelta $ enjoys quasi-optimality.
This immediately follows from Lemma~\ref{th:oracle-2}, applied with
$ \myd = \myq, \, \coc = \mygamma $, and $ \cod = \frac{1}{\mygamma} $.
However, such a parameter choice strategy is not implementable since
the function~$ \myphi $ is not available, in general. Thus we look
for feasible sets which contain $ \Mdelta $ and are as close as possible to
$ \Mdelta $.

For $ \coe, \cof, \cog > 0 $ fixed,
below we assume that
\begin{align}
\label{eq:alp0N-bounds}
0 < \alpha_0 \le \coe \delta^{1/\myb}, \qquad
\cof \le \alpha_N \le \cog,
\end{align}
which guarantees that a sufficiently large interval is covered by the set $ \Deltaset $
introduced in \eqref{eq:deltaset}.

Below we discuss several such balancing principles. These differ, e.g., in the number of
comparisons executed at each step, and it is seen from
the discussion in~Section~\ref{sec:miscellaneous} that by increasing the
number of comparisons we can decrease the error constant~$\cob$.

In each of the subsequent versions, the balancing is controlled
by a parameter~$ \mybeta > 0 $, therefore called \emph{balancing constant},
and it is assumed to satisfy
\begin{align}
  \label{eq:beta-def}
  \mybeta > 1 + \myqpowinv,
\end{align}
where~$q>1$ is the spacing parameter from~(\ref{eq:geometric_sequence}).
\paragraph*{{\bf First version}}
Here we consider
\begin{align} \label{eq:Hdelta-def}
\Hdelta := \insetla{ \alpha_k \in \Deltaset :
\normx{ \ualpdel[\alpha_j] - \ualpdel[\alpha_{j-1}] } \le \mybeta \frac{\delta}{\mylambda(\alpha_{j-1})}
\textup{ for any } 1 \le j \le k
}.
\end{align}

In order to find the maximum value~$\alpdel = \max \Hdelta$, we shall
start from~$k=0$, and increase~$k$ until
$\normx{\ualpdel[\alpha_{k+1}] - \ualpdel[\alpha_{k}] } > \mybeta        \frac{\delta}{\mylambda(\alpha_{k})} $ is satisfied for the first time, and take
$ \alpdel = \alpha_k $ then.
In the exceptional case that there is no such $ k \le N-1 $, it terminates with $ k = N $.
There is no need to compute the candidate approximations
$ \ualpdel[\alpha_{k+1}], \ualpdel[\alpha_{k+2}], \ldots, \ualpdel[\alpha_{N}] $
for this version.
In the present paper, our focus will be on this version.

\paragraph*{{\bf Standard version}}
The standard version of the balancing principle is related to the set
\begin{align} \label{eq:Hdeltab-def}
\Hdeltab := \insetla{ \alpha_k \in \Deltaset :
\normx{ \ualpdel[\alpha_k] - \ualpdel[\alpha_{j}] } \le \mybeta \frac{\delta}{\mylambda(\alpha_{j})}
\textup{ for any } 0 \le j < k
},
\end{align}
and it uses the maximum value~$\alpdel = \max \Hdeltab$ as regularizing parameter, cf., e.g.,
~\cite{Mathe06,PerSchock05,Pricop19}.
In order to find the maximum value~$\alpdel = \max \Hdelta$, one may
start from~$k=N$, and decrease~$k$ until the condition considered in \eqref{eq:Hdeltab-def} is satisfied for the first time.

This version of the balancing principle requires more comparisons than the first version introduced above, but on the other hand it allows to reduce the error constant. More details on the latter issue are given in Section~\ref{sec:miscellaneous}.
\paragraph*{{\bf  A third version}}
Finally we consider a variant of the standard version given through
\begin{align} \label{eq:Hdeltac-def}
\Hdeltac := \insetla{ \alpha_k \in \Deltaset :
\normx{ \ualpdel[\alpha_i] - \ualpdel[\alpha_{j}] } \le \mybeta \frac{\delta}{\mylambda(\alpha_{i})}
\textup{ for any } 0 \le i < j \le k
},
\end{align}
%
%
and consider the maximum value~$\alpdel = \max \Hdeltac$ as regularizing parameter.
This variant is considered in \cite{Plato[17]} in a special framework.
In order to find the maximum value~$\alpdel = \max \Hdeltac$, we shall
start from~$k=0$, and increase~$k$ until, for the first time, the condition
$\normx{\ualpdel[\alpha_{k+1}] - \ualpdel[\alpha_{j}] } > \mybeta        \frac{\delta}{\mylambda(\alpha_{j})} $ is satisfied for some index $ 0 < j \leq k $, and take $ \alpdel = \alpha_k $ then.
In the exceptional case that there is no such index~$ k \le N-1 $, the algorithm terminates with $ k = N $.
Only the candidate approximations
$ \ualpdel[\alpha_{0}],  \ualpdel[\alpha_{1}], \ldots,
\ualpdel[\alpha_{k}] $ have to be computed in the course of this
procedure.

Typically one expects $ \Hdeltab = \Hdeltac $, and one can show that this identity in fact holds, e.g., for Lavrentiev's method for solving linear, symmetric, positive semidefinite ill-problems. However, in general $ \Hdeltac \subset \Hdeltab $ can be guaranteed only, and the set $ \Hdeltab $ man have gaps in $ \Deltaset $. Under such general circumstances,
the standard version of the balancing procedure requires the computation of all elements
$ \ualpdel[\alpha_{0}],  \ualpdel[\alpha_{1}], \ldots, \ualpdel[\alpha_{N}] $
and consequently has a larger cost than the third version.
\begin{remark}
  In its original form, as introduced by \Lepski, the principle is based
on $\mygamma~=~1$ for the set~$\Mdelta$ from~(\ref{eq:mdelta}). In
this case the standard technique only allows choices $ \mybeta \ge 4
$.
Numerical experiments show that sometimes smaller balancing constants $ \mybeta $ produce better results.
We note that in the present paper, we verify quasi-optimality of the balancing principle for a range of balancing constants $ \mybeta $ which
is bounded from below by the number given in \eqref{eq:beta-def}.
Note that
condition~\eqref{eq:beta-def} even permits $ \mybeta $ close to $ 1 $
provided that $ \myq $ is chosen large. The latter case, however,
corresponds to a maybe undesirable coarse grid $ \Delta_\delta $. In
addition, it leads to a large error constant, as is shown below,
cf.~Proposition~\ref{th:c2_case_i} and the discussion following that
proposition.
\end{remark}

Finally we mention that within the present context, the analog of
Leonov's proposal would read as
\begin{equation*}
  \norm{ x_{\alpha_{i}}^{\delta} - x_{\alpha_{i-1}}^{\delta}}_{X} +
  \frac{\delta}{\mylambda(\alpha_{i-1})} \longrightarrow \min.
\end{equation*}
We shall establish the quasi-optimality for the first variant with
some details. The corresponding proofs for the other variants are
similar, and hence omitted. The quasi-optimality of Leonov's approach
is not clear for the present context of nonlinear ill-posed problems.

%
We start with the following observation.
\begin{lemma}
\label{th:msubh}
Suppose that the error bound~\eqref{eq:esti-error-general} holds true, where $ \mylambda(\alpha) $ is of the form \eqref{eq:lambda-def}.
In addition, let \eqref{eq:deltaset}, \eqref{eq:geometric_sequence}
and \eqref{eq:beta-def} be satisfied.
Let the tuning parameter~$ 0 <  \mygamma \leq 1$ used in the definition of the set $ \Mdelta $, cf.~\eqref{eq:mdelta}, be chosen sufficiently small such that
$ (\mygamma+1)(1+\myqpowinv) \le \mybeta $ holds. Then we
have
%
$ \Mdelta \subset \Hdelta $.
\end{lemma}

\begin{proof}
Let $ \alpha = \alpha_k \in \Mdelta $ and $ 1 \le j \le k $. Then $ \alpha_j, \, \alpha_{j-1} \in \Mdelta $, and thus
\begin{align*}
&
\norm{ \ualpdel[\alpha_j] - \ualpdel[\alpha_{j-1}] }_{X} \le
\norm{ \ualpdel[\alpha_j] - \ust }_{X}
+
\norm{ \ualpdel[\alpha_{j-1}] - \ust }_{X}
\\
& \quad \le
\myphi(\alpha_j) + \frac{\delta}{\mylambda(\alpha_j)}
+
\myphi(\alpha_{j-1}) + \frac{\delta}{\mylambda(\alpha_{j-1})}
\\
& \quad \le
(\mygamma+1)\Big(\frac{1}{\mylambda(\alpha_j)} +
\frac{1}{\mylambda(\alpha_{j-1})}\Big) \delta
=
(\mygamma+1)\lrc{1+\myqpowinv}
\frac{\delta}{\mylambda(\alpha_{j-1})}
\\
& \quad \le
\mybeta \frac{\delta}{\mylambda(\alpha_{j-1})}.
\end{align*}
Thus, each~$\alpha\in \Mdelta$ obeys the estimate
in~(\ref{eq:Hdelta-def}), and the proof is complete.
\end{proof}

Lemma \ref{th:msubh} and the considerations at the end of Section \ref{quasiopt}
give rise to the following a posteriori choice of the parameter $ \alpha = \alpdel $:
\begin{align}
\label{eq:lepskii}
\alpdel = \max \Hdelta.
\end{align}
\begin{theorem}
\label{th:lepskii-quasiopt}
Suppose that the error bound~\eqref{eq:esti-error-general} holds true, where $ \mylambda(\alpha) $ is of the form \eqref{eq:lambda-def}.
Let
\eqref{eq:deltaset}, \eqref{eq:geometric_sequence},
\eqref{eq:alp0N-bounds} and \eqref{eq:beta-def}
be satisfied.
Then the balancing principle~\eqref{eq:lepskii} is quasi-optimal.
%
\end{theorem}

\begin{proof}
The proof will distinguish three cases, and as a preparation we first prove
the following assertion.
For any $ \alpha \in \Deltaset $ with $ \alpha \le \alpdel $, we have
\begin{align}
\label{eq:mainerroresti-1}
\normx{\ualpdel[\alpdel] -\ust} \le
\myphi(\alpha) +
\coh \frac{\delta}{\mylambda(\alpha)},
\end{align}
where $ \coh = 1+ \frac{\mybeta}{1- \myqpowinv} $.
In fact, there are indices~$0 \leq k \leq N$ and $0\leq m \leq N-k$, such that $ \alpha = \alpha_k $ and
$ \alpdel = \alpha_{k+m} $. We can bound
\begin{align*}
& \normx{\ualpdel[\alpdel] -\ust}
\le
\normx{\ualpdel[\alpha] -\ust} +
\sum_{j=0}^{m-1}
\normx{\ualpdel[\alpha_{k+j+1}] - \ualpdel[\alpha_{k+j}] }
\\
& \quad \le
\myphi(\alpha) + \frac{\delta}{\mylambda(\alpha)}
+ \mybeta \delta \sum_{j=0}^{m-1} \frac{1}{\mylambda(\alpha_{k+j})}
\\
& \quad = \
\myphi(\alpha) + \frac{\delta}{\mylambda(\alpha)}
+ \frac{\mybeta \delta}{\mylambda(\alpha)}  \sum_{j=0}^{m-1} \myqpowinvj
\le
\myphi(\alpha)
+ \Big(1+ \frac{\mybeta}{1- \myqpowinv}\Big) \frac{\delta}{\mylambda(\alpha)}.
\end{align*}
This proves~(\ref{eq:mainerroresti-1}) with constant~$\coh$ as given.
%
%
%
%
We turn to the main part of the proof.
Assume that~$\mygamma>0$ is chosen as in
Lemma~\ref{th:msubh}. Clearly, $\Mdelta \subset \Deltaset$.
\begin{description}
\item[Case (i)]
  ($\Mdelta \neq \varnothing\   \text{and}\  \Mdelta \neq \Deltaset$)

The property $\Mdelta \neq \varnothing $ allows to consider
\begin{align}
\alptil:= \max \Mdelta.
\label{eq:mainerroresti-2b}
\end{align}
From the definition of $ \Mdelta $ and the assumption $ \Mdelta \neq \Deltaset $,
we obtain
$
(\myphi \mylambda)(\alptil) \le \mygamma \delta \le (\myphi \mylambda)(\myq\alptil)
$.
Thus, by Lemma \ref{th:oracle-2}, the parameter $ \alptil $ satisfies an oracle           inequality of the form \eqref{eq:oracle}.
In addition, by Lemma~\ref{th:msubh} we have that~$\Mdelta\subset H_{\delta}$, such
that the inequality $ \alptil \le \alpdel $ holds. Quasi-optimality of $ \alpdel $ under  the current situation now immediately follows from
the error estimate \eqref{eq:mainerroresti-1} applied with $ \alpha = \alptil $.
\item[Case (ii)]($\Mdelta = \Deltaset$)
%

For $ \alptil $ given by \eqref{eq:mainerroresti-2b}, this in fact means
$ \alptil = \alpdel = \alpha_N $ and thus
\begin{align}
(\myphi \mylambda)(\alpdel) \le \mygamma \delta.
\label{eq:mainerroresti-3}
\end{align}
For the lower bound of $ (\myphi \mylambda)(\alpdel) $,
we make use of $ \alpdel \ge \cof $ which implies that
$ \myphi(\alpdel) \ge \myphi(\cof) $ as well as
$ \mylambda(\alpdel) \ge \mylambda(\cof) = \sfrac{\cof^\myb}{\mykappa} $. We therefore arrive at
\begin{align}
\label{eq:mainerroresti-4}
\delta \le \cod (\myphi \mylambda)(\alpdel)
\quad \textup{ with }
\cod = \frac{\delta_0 \mykappa}{\myphi(\cof) \cof^\myb}.
\end{align}
This implies that $ \alpdel $ satisfies an estimate of the form
\eqref{eq:oracle-2}
and thus is quasi-optimal.
This completes the considerations of the case (ii).

\item[Case (iii)]($\Mdelta = \varnothing$)
    %

  In this case we may consider~$ \alptil := \alpha_0 $. This by~\eqref{eq:alp0N-bounds} means $ \alptil = \alpha_0 \le \coe \delta^{1/\myb} $,
and thus $ \mylambda(\alptil) \le (\coe^{\myb}/\mykappa) \delta $
and $ \myphi(\alptil) \le \myphi(\coe \delta_0^{1/\myb}) $.
Then, by the definition of $ \Mdelta $, we have
\begin{align}
\label{eq:mainerroresti-7}
\mygamma \delta \le (\myphi\mylambda)(\alptil)
\le
\frac{\coe^{\myb} \myphi(\coe \delta_0^{1/\myb}) }{\mykappa} \delta.
\end{align}
This implies that $ \alptil $ satisfies an estimate of the form
\eqref{eq:oracle-2} and thus also the oracle inequality \eqref{eq:oracle}. As in case~(i),
employing estimate~(\ref{eq:mainerroresti-1}), we
deduce an estimate of the   form \eqref{eq:quasioptimal} for $ \alpdel $ for this particular case.
\end{description}
The proof of the theorem is thus completed.
%
%
\end{proof}
%
%
\begin{remark}
  We stress that the case~(i) considered in the above proof is
  prototypical.
For, if~the maximum noise level $ \delta_0 $ is sufficiently small, the cases (ii) and (iii) cannot occur.
This follows from the estimates~\eqref{eq:mainerroresti-3} and \eqref{eq:mainerroresti-7},
which lead to contradictions then, respectively.
However, larger levels $ \delta_0 $ give
rise for the cases~(ii) and~(iii), respectively.
\end{remark}

\begin{remark}
\label{th:standard_version_quasiopt}
Lemma \ref{th:msubh} and~Theorem \ref{th:lepskii-quasiopt} 
also hold for
the other two balancing principles given by the sets $ \Hdeltab $ and $ \Hdeltac $, respectively.
More precisely, the same range of balancing constants $ \mybeta $, cf.~\eqref{eq:beta-def}, and tuning parameters $ \mygamma $ may be used. In
the wording of Lemma  \ref{th:msubh}, only $ \Hdelta $ has to be replaced by $ \Hdeltac $ and $ \Hdeltab $, respectively.
In \eqref{eq:mainerroresti-1} in the proof of Theorem~\ref{th:lepskii-quasiopt}, the constant $ \coh $ may be reduced to $ \coh = 1 + \mybeta $, which in fact has an impact on the corresponding error constant $ \cob $, cf.~the discussion in Section~\ref{sec:miscellaneous} below.
\end{remark}

%
%
%
%
%

\subsection{Discussion}
\label{sec:miscellaneous}

We shall discuss several aspects concerning the balancing principles.

\paragraph*{\sl Comparison of the three considered variants of the
  balancing principle}
We continue with a comparison of the considered variants of the balancing principle. Since we have
\begin{align*}
\Mdelta \subset \Hdeltac \subset \Hdelta,
\qquad \Mdelta \subset \Hdeltac \subset \Hdeltab,
\end{align*}
the latter balancing principle, which is related to the set $ \Hdeltac $, seems to be superior to the other versions. In fact, the set $ \Hdeltac $ it closer to the oracle set $ \Mdelta $ than the other two sets $ \Hdelta $ and $ \Hdeltab $. In addition, the latter version related to the set $ \Hdeltac $ requires less computational complexity, since the number of $ \ualpdel $ to be computed does not exceed the related number for the other versions. Note that for
the classical balancing principle related to the set $ \Hdeltab $,
one always has to compute $ \ualpdel $ for each $ \alpha \in \Deltaset $.
\paragraph*{\sl Oracle property of the parameter choices}
It may be of interest to consider quasi-optimality-type estimates without assuming
\eqref{eq:alp0N-bounds}, in particular, the minimal
value~$\alpha_0$, and the maximal~$\alpha_N$ are mis-specified.
The following result is obtained as a corollary of
Theorem \ref{th:quasioptimal} and its proof.
\begin{corollary}
\label{th:quasioptimal-mod}
For any of the three considered variants of the balancing principle,
we have
\begin{align*}
\normx{\ualpdel[\alpdel] -\ust} \le
\cob \inf_{\alpha_0 \le \alpha \le \alpha_N} \lrb{
\myphi(\alpha) + \frac{\delta}{\mylambda(\alpha)}
},
\end{align*}
%
where $ \cob > 0 $ denotes some finite constant.
\end{corollary}

\begin{proof}
We consider the first balancing principle only. The proofs for the
other two balancing principles are quite similar, and are left to the reader.
Below, for different situations, we verify estimates of the form
\begin{align}
\label{eq:quasioptimal-mod-1}
\normx{\ualpdel[\alpdel] -\ust} \le
\cob \inf_{\alpha \in \myI} \lrb{
\myphi(\alpha) + \frac{\delta}{\mylambda(\alpha)}
},
\end{align}
with appropriate intervals $ \myI $ and constants $ \cob $, respectively.
This in fact follows by a careful inspection of the proofs of
Lemma \ref{th:oracle-2} and Theorem~\ref{th:lepskii-quasiopt}. In the following considerations,
$ \mygamma $ denotes a constant satisfying the conditions of Lemma~\ref{th:msubh},
and for the meaning of the constant $ \coh $, we refer to \eqref{eq:mainerroresti-1}.

For case (i) considered in~Theorem \ref{th:lepskii-quasiopt}, i.e., $  \Mdelta \neq \varnothing $ and
$ \Mdelta \neq \Deltaset $, we have \eqref{eq:quasioptimal-mod-1} with
$ \myI = (0,\infty) $ and
$\cob = \myqpow\frac{\mygamma + \coh}{\mygamma} $.

For case (ii) in that theorem, i.e., $ \Mdelta = \Deltaset $, we have
$\alpdel = \alptil = \alpha_N $ and
$ \myphi(\alpdel) \le \mygamma \frac{\delta}{\mylambda(\alpdel)} $. The first part of the proof of Lemma~\ref{th:oracle-2}, applied with  $ \myd = 1 $, then gives
\eqref{eq:quasioptimal-mod-1} with
$ \myI = (0,\alpha_N] $ and $ \cob = \mygamma + \coh $.

Finally, for case (iii) considered in the theorem, i.e., $ \Mdelta = \varnothing $, we have $ \alptil = \alpha_0 $ and
$ \myphi(\alptil) \ge \mygamma \frac{\delta}{\mylambda(\alptil)} $. The second part of the proof of Lemma~\ref{th:oracle-2} then gives
\eqref{eq:quasioptimal-mod-1} with
$ \myI = [\alpha_0, \infty) $ and~$ \cob = \frac 1 {\mygamma}(\mygamma + \coh) $.
A combination of those three cases finally gives the statement of the corollary.
\end{proof}

The assertion of Corollary~\ref{th:quasioptimal-mod} may be considered
as \emph{oracle type}: If the range of
parameters~$[\alpha_{0},\alpha_{N}]$ is not specified correctly, then
the chosen parameter~$\alpha_{\ast}$ is, up to the constant~$\cob$, at
least as good as the best value within the specified range.
\paragraph*{\sl Controlling the error constant}
%
The following proposition specifies the error constant~$ \cob $ for
each of the considered balancing principles.
\begin{proposition}
\label{th:c2_case_i}
Let the maximum noise level $ \delta_0 $ be sufficiently small, and in
addition, let~\eqref{eq:geometric_sequence}, 
\eqref{eq:alp0N-bounds} and \eqref{eq:beta-def} be satisfied.
Then the error constant may be chosen as
\begin{align}
\cob =  \myqpow \frac{\mygamma + \coh}{\mygamma},
\label{eq:c2-choice}
\end{align}
where~$ \mygamma \leq \min\{\frac{\mybeta}{1+\myqpowinv}-1, 1\big\} $.
In addition,
$\coh:= 1 + \frac{\tau_{L}}{1 - \myqpowinv}$ for the
balancing principle~(\ref{eq:Hdelta-def}), and~$\coh:= 1 + \tau_{L}$ for the
versions from~(\ref{eq:Hdeltab-def}) and~(\ref{eq:Hdeltac-def}), respectively.
\end{proposition}
\begin{proof}
Under the given assumptions on $ \delta_0 $,
case (i) in the proof of Theorem~\ref{th:lepskii-quasiopt} applies,
i.e., we have
$\Mdelta \neq \varnothing $ and $ \Mdelta \neq \Deltaset$ there.
For~$ \alptil $ as in \eqref{eq:mainerroresti-2b}, an application of
the estimate in~(\ref{eq:mainerroresti-1}) for~$\alpha=\alptil$, and a careful
inspection of the proof of Lemma \ref{th:oracle-2} gives
\begin{align}
\normx{\ualpdel[\alpdel] -\ust}
& \le
\max\Big\{ \Big(1 + \frac{\myqpow\coh}{\mygamma}\Big) \myphi(\alpha), \
\myqpow (\mygamma + \coh) \frac{\delta}{\mylambda(\alpha)} \Big\}
\label{eq:c2_case_i_a}
\\
  & \le \myqpow \frac{\mygamma + \coh}{\mygamma}
\Big(\myphi(\alpha) + \frac{\delta}{\mylambda(\alpha)} \Big)
\qquad (\alpha > 0).
\nonumber
\end{align}
For the balancing principle~(\ref{eq:Hdelta-def}) this was shown to hold
for~$\coh= 1 + \frac{\tau_{L}}{1 - \myqpowinv}$. For the balancing principles
from~(\ref{eq:Hdeltab-def}) and~(\ref{eq:Hdeltac-def}) the reasoning
in Theorem~\ref{th:lepskii-quasiopt} simplifies, and the
bound~(\ref{eq:mainerroresti-1}) holds with~$\coh:= 1 + \tau_{L}$.
This completes the sketch of this proof.
\end{proof}

We note that the special form of $ \mygamma $ considered in
Proposition~\ref{th:c2_case_i} is caused by the requirement made in Lemma~\ref{th:msubh}.

We next discuss the optimal choice of the parameters used in the
balancing principle \eqref{eq:lepskii} to minimize the error constant $ \cob
$.
First, we consider the spacing parameter~$q>1$ to be fixed.
Thus, in order to minimize~$\cob$ we need to consider~$\frac{\mygamma
  + \coh}{\mygamma}$ only. The constant~$\coh$, as a function
of~$\tau_{L}$ is monotone, such that the smallest possible value
of~$\tau_{L}$ minimizes~$\coh$, and, taking into account the
requirements in Lemma~\ref{th:msubh}, we let~$\tau_{L}(\myqpow):= (\mygamma
+1)(1 + \myqpowinv)$. Now we need to distinguish the values for~$\coh$
as indicated in Proposition~\ref{th:c2_case_i}. For the first
balancing principle, based on~(\ref{eq:Hdelta-def}), we find that
\begin{equation*}
  \cob = \myqpow \frac{\mygamma + \coh}{\mygamma}
=
\frac{\mygamma + 1}{\mygamma}\frac{2 \myqpow}{ 1 - \myqpowinv}
\geq \frac{ 4\myqpow}{ 1 - \myqpowinv},
\end{equation*}
the latter being achieved for~$\mygamma=1$. This can further be
optimized with respect to the spacing parameter~$q$, and it is minimized
for~$\myqpow = \myqpow_{\text{opt}}:=2$.
With these specifications we find that
\begin{equation*}
  \tau_{L,\text{opt}}= 3,\quad \text{and}\quad c_{2,\text{opt}}= 16.
\end{equation*}
%
We next consider the size of the error constants of the other two balancing principles
related with the sets given by \eqref{eq:Hdeltab-def} and
\eqref{eq:Hdeltac-def}, respectively. In either case, the error constant $ \cob $ is again of the form
\eqref{eq:c2-choice}, with $ \coh = 1 + \mybeta $.

Thus, for $ \mybeta =
(\mygamma +1)(1+\myqpowinv) $ we have that
\begin{align*}
\cob = \myqpow \frac{\mygamma + \coh}{\mygamma}
= \myqpow \frac{\mygamma +1}{\mygamma}(2 + \myqpowinv).
\end{align*}
Again, this is minimized for~$\mygamma:= 1$, and it gives
\begin{equation}\label{eq:c2-choice-d}
c_{2,\text{opt}} =2 \lr{2 \myqpow + 1}
\end{equation}
with corresponding~$\tau_{L,\text{opt}}= 2 (1+\myqpowinv) < 4 $.
This means that the error constant becomes smaller as the grid $ \Deltaset $ becomes finer, with
$ \cob
\to 6 $ as $ \myq \to 1 $.
For the best grid independent choice of the
  parameter~$\tau_{L}$, we find that~$\tau_{L}=4$,  and hence  we recover the original Lepski\u\i{}
  principle with constant~$\cob=6q^{\varkappa}$.


\begin{remark}
The quasi-optimality results for all three methods considered in the present work can also be written in the frequently used form
\begin{align}
\label{eq:c2-choice-e}
\normx{\ualpdel[\alpdel] -\ust} \le \cob \myphi(\alptil),
\end{align}
where the parameter $ \alptil > 0 $ satisfies $ \myphi(\alptil) = \frac{\delta}{\lambda(\alptil)} $, and the error constant $ \cob $ is given by \eqref{eq:c2-choice}.
This can be seen by considering estimate \eqref{eq:c2_case_i_a} in the proof of Proposition \ref{th:c2_case_i}.
Note that we have $ \myphi(\alptil) \le
\inf_{\alpha > 0} \lrc{\myphi(\alpha) + \frac{\delta}{\mylambda(\alpha)}}
\le 2 \myphi(\alptil) $, so quasi-optimality
is in fact equivalent to
\eqref{eq:c2-choice-e} for some constant $ \cob $.

For the standard balancing principle \eqref{eq:Hdeltab-def}, utilized with the traditional balancing constant $ \mybeta = 4 $, the error constant
in estimate \eqref{eq:c2-choice-e} takes the form $ \cob = 6 \myqpow $,
which is a well-known result, cf., e.g., \cite{Mathe06,PerSchock05}.
The above discussion shows that the error constant $ \cob $
in \eqref{eq:c2-choice-e} can be reduced to the form \eqref{eq:c2-choice-d} by choosing $ \mybeta  $ somewhat smaller.
\end{remark}

\subsection{Specific impact on oversmoothing penalties}
\label{sec:impact}

After the presentation of various facets of the general theory for
balancing principles in the preceding paragraphs of this section, we
return to the specific situation
of oversmoothing penalties as outlined in Sections~\ref{intro} and~\ref{TRO-props}. To characterize the impact of the general theory on that situation, we recall the error estimate (\ref{eq:tikh-overesti-error})
with the specific index function  $ \mylambda(\alpha) =  \frac{1}{\mykappa}\alpha^{a/(2a+2)} $ and with the specific constant $ \mykappa = \max\inset{1,\frac{2}{c_a}} $.
 
In Examples~\ref{xmpl:power} and~\ref{xmpl:log} we have explicitly
described a priori parameter choice rules as well as the discrepancy
principle as an a posteriori choice rule. For the nonlinear inverse
problem~(\ref{eq:opeq}) at hand, the corresponding convergence
rates are given in~(\ref{eq:typeH2}) for the H\"older case, and~(\ref{eq:typeL}) for
the logarithmic case. Here, we complement those rate results by analog
assertions for the balancing principles.  The results are based on the
quasi-optimality of the balancing principles under consideration, 
in connection with the first part of Proposition~\ref{th:rate-guarantee}.

In case that no explicit smoothness for the solution of the nonlinear inverse problem \eqref{eq:opeq} is available, we note that any quasi-optimal rule for Tikhonov regularization with oversmoothing penalty yields convergence. For $ B^{-1} $ compact, this can be seen by consulting Example \ref{xmpl:no-smoothness}
and the second part of Proposition \ref{th:rate-guarantee}.

We briefly summarize the impact of the theory of the first
balancing principle~\eqref{eq:lepskii}, when applied to nonlinear Tikhonov
regularization with oversmoothing penalty, by the following corollary.
Note that the assertions of the corollary can be formulated in an
analog manner for the other two balancing principles, and we refer to Remark~\ref{th:standard_version_quasiopt} above. 

\begin{corollary}
Consider nonlinear Tikhonov regularization with oversmoothing penalty as introduced in Sections~\ref{intro} and~\ref{TRO-props}, with the regularization parameter
$ \alpstar $ determined by the balancing principle \eqref{eq:lepskii} under the required conditions (\ref{eq:alp0N-bounds}) and  (\ref{eq:beta-def}).
For H\"older-type smoothness (\ref{eq:typeH1}) as considered in Example \ref{xmpl:power}, with $ 0 < p \le 1 $, one has
\begin{align*}
\|x^\delta_{\alpha_{\ast}}-\xdag\|_X=\mathcal O\lrc{\delta^{\frac p{a+p}}} \quad \text{as  } \ \delta \to 0.
\end{align*}
Similarly, for logarithmic source conditions as considered in Example~\ref{xmpl:log},
we obtain logarithmic rates
\begin{align*}
\|x^\delta_{\alpha_{\ast}}-\xdag\|_X= \mathcal O\lrc{\log^{-\mymu}(1/\delta)} \quad \text{as } \ \delta \to 0.
\end{align*}
\end{corollary}
\begin{proof}
This is an immediate  consequence of quasi-optimality of the balancing principle \eqref{eq:lepskii},  and of the convergence rate results
in Examples~\ref{xmpl:power} and \ref{xmpl:log} in connection with the first part of Proposition~\ref{th:rate-guarantee}.
\end{proof}
We explicitly highlight the following fact, intrinsic in the proof: If the
  parameter~$\alpha_{\ast}(\delta)$, obtained by the a priori choice,
  cf.~Examples~\ref{xmpl:power}--\ref{xmpl:no-smoothness}, is in the
  interval~$[\alpha_{0}(\delta),\alpha_{N}(\delta)]$, then the corresponding
  convergence rates for the parameter choice according to the
  balancing principle are valid. Otherwise
  Corollary~\ref{th:quasioptimal-mod} applies.

As already noticed, the study~\cite{Pricop19} by Pricop-Jeckstadt is
close to our approach on balancing principles. However, it does not
include the case of oversmoothing penalties, a gap which is closed
here. Despite the fact that the nonlinearity requirements
of~\cite{Pricop19} are slightly different, the main difference lies in
the following fact: The proofs (for the non-oversmoothing case) in ibid.
are based on an error decomposition into a noise amplification error, and a bias that occurs when the data
are noise-free. This technique fails in the oversmoothing case, where
instead a certain auxiliary element is used. 
\section{Exponential growth model: properties and numerical case
  study} \label{example}
For a case study we shall collect the theoretical properties of the
exponential growth model, first presented in~\cite[Section
3.1]{Groe93}. More details about properties of the nonlinear forward
map~$F$ as in~(\ref{eq:F}), below, can be found in \cite{Hof98}.  Then we
highlight different behavior for the reconstruction in the
oversmoothing and non-oversmoothing cases, respectively.
\subsection{Properties}\label{properties}
For analytical and numerical studies we are going to exploit the exponential growth model
\begin{equation} \label{eq:ode}
y'(t)=x(t)\,y(t) \quad (0 \le t \leq 1), \qquad y(0)=1,
\end{equation}
considered in the Hilbert space $L^2(0,1)$. The inverse problem
consists in the identification of the square-integrable time-dependent
function~$x(t)\;(0 \le t \le 1)$ in (\ref{eq:ode}) from noisy data~$\yd\in L^2(0,1)$
of the solution $y(t)\;(0 \le t \le 1)$ to the corresponding initial value O.D.E. problem.
 In this context, we suppose a deterministic noise
 model~$\|\yd-y\|_{L^2(0,1)} \le \delta$ with noise level $\delta>0$.
This identification problem can be written in form of an operator equation (\ref{eq:opeq}) with the nonlinear forward operator
\begin{equation}\label{eq:F}
[F(x)](t)= \exp \; \left(\int  _{0}^{\,t} x(\tau) d\tau\right)\qquad (0 \leq t \leq 1)
\end{equation}
mapping in $L^2(0,1)$ with full domain $\domain(F)=L^2(0,1)$.

Evidently, $F$ is globally {\sl injective}. One can also show  on
the one hand that $F$ is {\sl weakly sequentially continuous} and  on
the other hand that~$F$ is {\sl Fr\'echet differentiable} everywhere. It possesses for all $x \in L^2(0,1)$ the Fr\'echet derivative
$F^\prime(x)$, explicitly given as
\begin{equation} \label{eq:Fprime}
[F^\prime (x) h](t)=[F(x)](t) \int_0^{\,t} h(\tau)d\tau  \quad (0 \le t \le 1,\;\; h \in L^2(0,1)).
\end{equation}
This Fr\'echet derivative is a compact linear mapping in $L^2(0,1)$,
because it is a composition $F^\prime(x)=M \circ J$ of the bounded
linear multiplication operator $M$ mapping in $L^2(0,1)$ defined  as
$$
[Mg](t)=[F(x)](t)\,g(t)\;(0 \le t \le 1),
$$
and the compact linear integration operator $J$ mapping in $L^2(0,1)$ defined as
\begin{equation}\label{eq:J}
[Jh](t)=\int_0^{\,t} h(\tau)d\tau  \quad (0 \le t \le 1).
\end{equation}
We mention that the continuous multiplier function $F(x)$ in $M$ is bounded below and above by finite positive values due to
\begin{equation} \label{eq:Mbounds}
\exp(-\|x\|_{L^2(0,1}) \le [F(x)](t) \le \exp(\|x\|_{L^2(0,1)}) \  (0
\le t \le 1),
\end{equation}
 for~$x \in L^2(0,1)$, and hence~$F(x) \in L^\infty(0,1)$.

Let us denote by
  $$
\mathcal{B}_r(\xdag)=\{z \in
L^2(0,1):\,\|z-\xdag\|_{L^2(0,1)} \le r\},
$$
the closed ball with radius $r>0$ and center $\xdag$.

The following lemma highlights, that $F$ satisfies a nonlinearity
condition of tangential cone-type.
\begin{lemma} \label{lem:TCC}
For the nonlinear operator $F$ from (\ref{eq:F}), the inequality
\begin{multline} \label{eq:TCC}
  \|F(x)-F(\xdag)-F^\prime(\xdag)(x-\xdag)\|_{L^2(0,1)} \\
  \le \|x-\xdag\|_{L^2(0,1)}\,\|F(x)-F(\xdag)\|_{L^2(0,1)}
\end{multline}
is valid for all $x,\xdag  \in L^2(0,1)$. Consequently, we have for arbitrary but fixed $r>0$ and $\xdag \in L^2(0,1)$ the inequality
\begin{equation} \label{eq:LeftTCC}
\frac{1}{1+r}\,\|F^\prime(\xdag)(x-\xdag)\|_{L^2(0,1)} \le \|F(x)-F(\xdag)\|_{L^2(0,1)}
\end{equation}
 for all $x \in \mathcal{B}_r(\xdag)$. Moreover,   whenever
 $0<r<1$ we have that
\begin{equation} \label{eq:RightTCC}
\|F(x)-F(\xdag)\|_{L^2(0,1)} \le \frac{1}{1-r}\,\|F^\prime(\xdag)(x-\xdag)\|_{L^2(0,1)}
\end{equation}
 for all $x \in \mathcal{B}_r(\xdag)$.
\end{lemma}

\begin{proof}
By setting $\theta(t):=[J(x-\xdag)](t)\;(0 \le t \le 1)$ we have
$$[F(x)-F(\xdag)](t)=[F(\xdag)](t)\,(\exp(\theta(t))-1) $$
and
$$[F(x)-F(\xdag)-F^\prime(\xdag)(x-\xdag)](t)=[F(\xdag)](t)\,(\exp(\theta(t))-1-\theta(t)).$$
Then the general estimate
$$|\exp(\theta)-1-\theta| \le |\theta|\,|\exp(\theta)-1|, $$
which is valid for all $-\infty<\theta<+\infty$, leads to
$$
|[F(x)-F(\xdag)-F^\prime(\xdag)(x-\xdag)](t)| \le
|\theta(t)|\,|[F(x)-F(\xdag)](t)|,
$$
for all~$0 \leq t \leq 1$.
By using the Cauchy--Schwarz inequality this implies
$$|[F(x)-F(\xdag)-F^\prime(\xdag)(x-\xdag)](t)| \le
\|x-\xdag\|_{L^2(0,1)}, $$
again, for~$0 \leq t \leq 1$.
This yields the inequality (\ref{eq:TCC}).  Both
inequalities~(\ref{eq:LeftTCC}) and~(\ref{eq:RightTCC}) are immediate consequences
by applying the triangle inequality.
The proof of the lemma is complete.
\end{proof}

We are going to establish that the forward operator $F$ from
(\ref{eq:F}) obeys the nonlinearity condition~(\ref{eq:twosided}). To
this end we use the Hilbert scale model as introduced in
Section~\ref{sec:scales}. Here, the Hilbert scale $\{X_\tau\}_{\tau \in \mathbb{R}}$ with $X_0=X=L^2(0,1),$ $X_\tau=\domain(B^\tau)$ for $\tau>0$ and $X_\tau=X$ for $\tau<0$, is generated by the unbounded, self-adjoint, and positive definite linear operator
\begin{equation} \label{eq:altB}
B:=(J^*J)^{-1/2}
\end{equation}
induced by the integration operator $J$ from (\ref{eq:J}). The domain $\domain(B)$ of~$B$ is dense in $L^2(0,1)$ and its range $\mathcal{R}(B)$ coincides with $L^2(0,1)$. For each $\tau \in \mathbb{R}$ one can define the norm
$$\|x\|_\tau:=\|B^\tau x\|_{L^2(0,1)} \quad \mbox{defined for all} \quad x \in X_\tau.$$
The powers of $B$ are linked to the powers of $J$. By analyzing the Riemann--Liouville fractional integral operator $J^p$ for levels $p$
from the interval $(0,1]$ we have that
\[
X_p=\domain(B^{p})=\mathcal{R}((J^*J)^{p/2}) \quad \mbox{for}\quad 0<p \le 1.
\]
Due to~\cite[Lemma~8]{GorYam99} this gives the explicit representation
\begin{equation} \label{eq:fracrange}
X_p=\left\{\;\begin{array}{ccc} H^p(0,1) &\quad \mbox{for}\quad & 0<p<\frac{1}{2}\\ \{x \in H^{\frac{1}{2}}(0,1):\int\limits_0^1 \frac{|x(t)|^2}{1-t}dt<\infty \}&\quad\mbox{for}\quad& p=\frac{1}{2}
\\ \{x \in H^p(0,1):\, x(1)=0\}   &\quad \mbox{for} \quad& \frac{1}{2}<p \le 1 \end{array} \right.,
\end{equation}
where $H^p(0,1)$ denotes the corresponding fractional hilbertian Sobolev space.
Note that for $0<p<\frac{1}{2}$ the spaces~$X_p$ and the hilbertian Sobolev spaces $H^p(0,1)$ coincide, whereas for $p>\frac{1}{2}$ an additional homogeneous boundary condition occurs at the right end of the interval.

Using the Hilbert scale introduced above we have collected now all
ingredients for verifying an inequality chain of type
(\ref{eq:twosided}) with a degree $a=1$ of ill-posedness.

We start with
$$\|Jh\|_{L^2(0,1)}=\|(J^*J)^{1/2}h\|_{L^2(0,1)}=\|B^{-1}h\|_{L^2(0,1)}=\|h\|_{-1}$$
valid for all~$ h \in L^2(0,1) $, 
and we aim at applying Lemma~\ref{lem:TCC}. In this context, we set on the one hand
$k_0:=\exp(-\|\xdag\|_{L^2(0,1)})$, $K_0:=\exp(\|\xdag\|_{L^2(0,1)})$
and on the other hand $c_{1}:=\frac{k_0}{1+r}$ and $C_{1}:=
\frac{K_0}{1-r}$. Then we have $0< k_0 \le [F(\xdag)](t) \le K_0 < \infty$
from (\ref{eq:Mbounds}). By using formula~(\ref{eq:Fprime}) we obtain, for all  $x \in L^2(0,1)$ and $r>0$,
the estimate
\begin{equation} \label{eq:left}
c_{1}\|x-\xdag\|_{-1} = \frac{k_0}{1+r}\,\|J(x-\xdag)\|_{L^2(0,1)} \le \frac{1}{1+r}\,\|F^\prime(\xdag)(x-\xdag)\|_{L^2(0,1)},
\end{equation}
whenever~$x\in \mathcal{B}_r(\xdag)$.
In the same manner one deduces that for all  $x \in L^2(0,1)$, and
$0<r<1$,   the right-side estimate
$$
\frac{1}{1-r}\,\|F^\prime(\xdag)(x-\xdag)\|_{L^2(0,1)} \le
\frac{K_0}{1-r}\,\|J(x-\xdag)\|_{L^2(0,1)} =C_{1}\|x-\xdag\|_{-1}
$$
holds true whenever~$x\in \mathcal{B}_r(\xdag)$. By Lemma~\ref{lem:TCC} this yields the inequality chain
\begin{equation*} 
c_{1}\,\|x-\xdag\|_{-1} \le \|F(x)-F(\xdag)\|_{L^2(0,1)}\le C_{1}\,\|x-\xdag\|_{-1} \;\; \mbox{for}\; \; x \in \mathcal{B}_r(\xdag),
\end{equation*}
with~$0 < r < 1. $
This proves the assertion of the following proposition.
\begin{proposition}\label{pro:example}
For $X=Y=L^2(0,1)$ we consider the nonlinear operator $F$ from
(\ref{eq:F}). Its  domain~$\domain(F)$ is restricted to a closed ball
$\mathcal{B}_r(\xdag)$ around some element~$\xdag\in X$,
and with radius radius $r<1$.

Within  the Hilbert scale generated by the operator $B$ from
(\ref{eq:altB}), induced by the integration operator $J$ from
(\ref{eq:J}), the operator~$F$ obeys the nonlinearity condition~(\ref{eq:twosided}) with
$a=1$. The positive constants $c_1$ and~$C_1$
depend on $\xdag$ and $r$.
\end{proposition}

Next we discuss assertions on missing stability and well-posedness for the operator equation (\ref{eq:opeq}).

\begin{proposition}\label{pro:locallyill} For arbitrary solution $\xdag \in L^2(0,1)$
the operator equation (\ref{eq:opeq}) with $F:\domain(F)=X=L^2(0,1)
\to Y=L^2(0,1)$  from (\ref{eq:F})  is {\sl locally ill-posed}
(cf.~\cite[Definition~3]{HofPla18}). It is {\sl not stably solvable} at $y$ (cf.~\cite[Definition~1]{HofPla18}) for arbitrary right-hand element $y \in \mathcal{R}(F):=\{z \in L^2(0,1):\,z=F(\xi),\;\xi \in L^2(0,1)\}$  from the range of $F$.
\end{proposition}

\begin{proof}
One can easily show that, as a consequence of the compactness of $J$ in $L^2(0,1)$, the operator  $F:\domain(F)=X=L^2(0,1) \to Y=L^2(0,1)$  from (\ref{eq:F}) is {\sl strongly continuous} in
the sense of Definition~26.1~(c) from \cite{Zeidler90}. Thus weakly convergent sequences $x_n \rightharpoonup x_0$ in $L^2(0,1)$ imply
norm convergent image sequences such that $\lim_{n \to
  \infty}\|F(x_n)-F(x_0)\|_{L^2(0,1)}=0$. For any orthonormal system
$\{e_n\}_{n=1}^\infty$ in $L^2(0,1)$ and any radius $r>0$ we then have
$x_n=\xdag+r e_n \rightharpoonup \xdag, \; x_n \in
\mathcal{B}_r(\xdag)$, and $\lim_{n \to
  \infty}\|F(x_n)-F(x_0)\|_{L^2(0,1)}~=~0$. On the one hand this
proves the local ill-posedness at $\xdag$. On the other hand, since
$F$ is injective with the inverse operator $F^{-1}: \mathcal{R}(F)
\subset L^2(0,1) \to L^2(0,1)$ , the mapping~$F^{-1}$ cannot be continuous at $y=F(\xdag)$, which contradicts the
stable solvability of the operator equation at $y$. The proof is complete.
\end{proof}

Note that for injective forward operators stable solvability and continuity of the inverse operator coincide. Moreover note that in the Hilbert space $L^2(0,1)$ the strong continuity of $F$ implies that $F$ is compact (cf.~\cite[Proposition~26.2]{Zeidler90}).

\begin{remark}
Due to the local ill-posedness of $F$  defined by formula (\ref{eq:F}) and mapping from $X_0=X=L^2(0,1)$ to $Y=L^2(0,1)$ with the associated norm topologies, we have that for arbitrarily small $r>0$ there is no constant $c_0>0$ such that
$$c_0\,\|x-\xdag\|_{L^2(0,1)} \le \|F(x)-F(\xdag)\|_{L^2(0,1)} \;\; \mbox{for all}\; \; x \in  \mathcal{B}_r(\xdag). $$
However, if the weaker $X_{-1}$-norm $\|\cdot\|_{-1}=\|B^{-1} \cdot\|_{L^2(0,1)}=\|J \cdot\|_{L^2(0,1)} $  applies for the pre-image space of $F$, one can see from the estimates~(\ref{eq:LeftTCC}) and  (\ref{eq:left}) that for all $r>0$ there exists a constant
$c_{1}>0$ depending on $\xdag$ and $r$ such that
$$c_{1}\,\|x-\xdag\|_{-1} \le \|F(x)-F(\xdag)\|_{L^2(0,1)} \;\; \mbox{for all}\; \; x \in \mathcal{B}_r(\xdag), $$
which proves that (\ref{eq:opeq}) is {\sl locally well-posed} everywhere for that norm pairing. For the convergence theory of Tikhonov regularization in Hilbert scales in case of oversmoothing penalties, however,
this requires an a priori restriction of the domain $\domain(F)$ to bounded sets (balls), because (\ref{eq:twosided}) is originally needed in \cite{HofPla20} with respect to that example with the full domain $\domain(F)=L^2(0,1)$.

On the other hand, there exists no global constant $c_{1}>0$ depending only on $\xdag$ such that
$c_{1}\,\|x-\xdag\|_{-1} \le \|F(x)-F(\xdag)\|_{L^2(0,1)} $ for each $ x \in  L^2(0,1) $.
This follows, for any $ \xdag $, by considering, e.g., the functions
$$ x(t) = x_n(t) \equiv -n \ \textup{ for } \ n= 1,2, \ldots \ . $$
In fact, we then have~$[Jx_n](t) = -nt $
and $ [F(x_n)](t) = \exp(-nt) $, and thus
\begin{align*}
\|Fx_n \|_{L^2(0,1)}^2 = \frac{1}{2n} (1-e^{-2n}) \to 0,\intertext{but}
\|x_n\|_{-1} =\|Jx_n \|_{L^2(0,1)} = \frac{n}{\sqrt{3}} \to \infty,
\end{align*}

as $ n  \to \infty $.

The restriction to a small ball with radius~$r<1$ for the right-hand
inequality of (\ref{eq:twosided}), as
caused by the condition (\ref{eq:RightTCC}), is not problematic for the theory (cf.~\cite{HofMat18,HofPla20}). This part is only used by auxiliary elements that are close to $\xdag$ for sufficiently small
regularization parameters $\alpha>0$ whenever $\xdag$ is supposed to be an interior point of $\domain(F)$, which is trivial for  $\domain(F)=X$.
\end{remark}

The character of ill-posedness of the operator equation
(\ref{eq:opeq}) with $F$ from (\ref{eq:F}) can be illustrated at the
exact solution
\begin{equation} \label{eq:constantone}
\xdag(t) \equiv 1 \quad(0 \le t \le 1),
\end{equation}
which will be used in the numerical study below. In this case we have $[F(\xdag)](t)=\exp(t) \;(0 \le t \le 1)$. If we perturb this exact right-hand side by a continuously differentiable noise function $\eta(t)\;(0 \le t \le 1)$ with $\eta(0)=0$, then the pre-image of $F(\xdag)+\eta$ attains the explicit form
\begin{equation*}
x_\eta(t)= \frac{\exp(t)+\eta^\prime(t)}{\exp(t)+\eta(t)} \quad (0 \le t \le 1).
\end{equation*}
In particular, for $\eta(t)=\delta \,\sin(nt) \;(0 \le t \le 1)$ with multiplier $\delta>0$ and $\|\eta\|_{L^2(0,1)} \le \delta$ we have $x_\eta=x_n$ defined as
\begin{equation} \label{etasin}
x_n(t)= \frac{\exp(t)+n\,\delta\,\cos(nt)}{\exp(t)+\delta\,\sin(nt)} \quad (0 \le t \le 1),
\end{equation}
as well as $\|F(x_n)-F(\xdag)\|_{L^2(0,1)} \le \delta$ for all $n \in \mathbb{N}$ and $\delta>0$.
The next proposition emphasizes for $F$ from (\ref{eq:F}) that in spite of very small image deviations $\|F(x)-F(\xdag)\|_{L^2(0,1)}$ the corresponding error norm $\|x-\xdag\|_{L^2(0,1)}$ can be arbitrarily large.

\begin{proposition}\label{pro:oneexplode}
For arbitrarily small $\delta>0$ the pre-image set~$F^{-1}(\mathcal{B}_\delta(F(\xdag)))$ for $\xdag$ from (\ref{eq:constantone}) is not bounded in $L^2(0,1)$, i.e.~there exist sequences
$\{x_n\}_{n=1}^\infty \subset F^{-1}(\mathcal{B}_\delta(F(\xdag)))$ with $\lim_{n \to \infty}\|x_n\|_{L^2(0,1)}=+\infty$.
\end{proposition}

\begin{proof}
To this end we shall use the explicit sequence
$\{x_n\}_{n=1}^\infty$ from (\ref{etasin})
, for which the estimate
$$\|x_n\|^2_{L^2(0,1)}=\int_0^1  \frac{(\exp(t)+n\,\delta\,\cos(nt))^2}{(\exp(t)+\delta\,\sin(nt))^2}dt =\int_0^1 \frac{(1+n\,\delta\,\frac{\cos(nt)}{\exp(t)})^2}{(1+\delta\,\frac{\sin(nt)}{\exp(t)})^2}dt $$
$$\ge \frac{n^2\delta^2 \int_0^1 (\cos(nt))^2 dt }{e^2(1+\delta)^2}=  \frac{n^2\delta^2 }{e^2(1+\delta)^2}\left(\frac{1}{2}+\frac{\sin(n)\cos(n)}{2n} \right) $$
holds true. Consequently, we have $\|x_n\|_{L^2(0,1)} \ge \frac{n \delta}{2e(1+\delta)}$ for sufficiently large $n \in \mathbb{N}$, such that $\lim_{n \to \infty}\|x_n\|_{L^2(0,1)}=+\infty$.
\end{proof}
\subsection{Numerical case study}\label{numerics}
The following numerical case study operates in the setting of
Section~\ref{properties} and complements the theoretical
results. Therefore we minimize the Tikhonov functional of
type~\eqref{eq:tikhonov} with $\bar{x}=0$ and forward
operator~\eqref{eq:F} derived from the exponential growth
model. Recall that the inequality chain~\eqref{eq:twosided} holds with
degree of ill-posedness $a=1$. Further, set $X=Y=L^2(0,1)$. To obtain
the $X_1$-norm in the penalty, use $\|\cdot \|_1=\|\cdot\|_{H^1(0,1)}$
and additionally enforce the boundary condition $x(1)=0$ in accordance
with the construction of the Hilbert scale~\eqref{eq:fracrange}. In
all experiments we use the exact solution
$\xdag(t)=1; \; (0 < t \leq 1)$. As this particular exact solution is
smooth, but violates the boundary condition $\xdag(1) = 0$, we have
$\xdag \in X_p$ for all $0<p<\frac{1}{2}$. This means, a H\"older-type
source condition as outlined in Example~\ref{xmpl:power} holds.
A discretization level of $N=1000$ in the time domain is
used. The noise is then constructed by sampling one realization of
  a vector~$\xi=(\xi_{1},\dots,\xi_{1000})$ of $1000$ i.i.d. standard
  Gaussian random variables. This is then normalized to
  have~$\|\xi\|_{L^2(0,1)}=1$, and $\delta \xi$ added to the exact
  data~$y$.  For noise level~$\delta$ this yields
\begin{math}
  \|y-\yd\|_{Y} = \delta.
\end{math}
The minimization problem itself is solved using the
MATLAB\textregistered -routine \texttt{fmincon}. Integrals are
discretized using the trapezoid rule.

Further, a modification
of the first variant~\eqref{eq:Hdelta-def} of the balancing principle is
implemented.
Precisely, we modify $\Hdelta$ introduced in formula~(\ref{eq:Hdelta-def}) as
\begin{align}\label{eq:bal-mod-bernd}
\Hdelta^{mod} := \insetla{ \alpha_k \in \Deltaset :
\normx{ \ualpdel[\alpha_j] - \ualpdel[\alpha_{j-1}] } \le C_{BP}\, \frac{\delta}{\alpha_{j-1}^{a/(2a+2)}}
\textup{ for} 1 \le j \le k
}.
\end{align}
This is necessary, because the constant~$K_{2}$, which is required to be known
in~(\ref{eq:lambda-def}), is not available. So,  we use the
constant~$C_{BP}$ as a replacement for~$\tau_{L}K_{2}$, instead. Then we set $\alpha_*=\alpha_{BP}:=\max \Hdelta^{mod}.$

Since $x^\dag$ is known we can compute the
regularization errors $\|x_\alpha^\delta-x^\dag\|_X$. This can be
interpreted as a function of $\delta$ and justifies a regression for
the model function
\begin{equation}\label{eq:regression_model_x}
  \|x_\alpha^\delta-x^\dag\|_X\leq c_x \delta^{\kappa_x}.
\end{equation}
Similarly we estimate the asymptotic behavior of the selected
regularization parameter through the ansatz
\begin{equation}\label{eq:regression_model_alpha}
  \alpha\sim c_\alpha \delta^{\kappa_\alpha}.
\end{equation}
Both exponents $\kappa_x$ and $\kappa_\alpha$ as well as the
corresponding multipliers $c_x$ and $c_\alpha$ are obtained using a
least squares regression based on samples for varying $\delta$.

In a first case study we consider several constants $C_{BP}$ used in
the balancing principle, the results of which are displayed in
Table~\ref{tab:rates_BP}.  Recall the results of
Theorem~\ref{thm:general} as well as Example~\ref{xmpl:power}: if a
H\"older-type source condition holds, we expect
$\kappa_x=\frac{p}{a+p}$. As $a=1$ and $p \approx \frac{1}{2}$, this
corresponds very well with the numerical observations in the third
column of Table~\ref{tab:rates_BP}. Further recall the a priori
parameter choice $\alpha_*=\delta^{\frac{2a+2}{a+p}}$, which here
reads as~$\alpha_*=\delta^{\frac{8}{3}}$. The fifth column of
Table~\ref{tab:rates_BP} shows the resulting $\alpha-$rates for the
balancing principle. We therefore conclude that the resulting rates
coincide with this particular a priori choice.
\begin{table}[h!]
  \begin{center}
    \begin{tabular}{ p{1.5cm} p{1.5cm} p{1.5cm} p{1.5cm} p{1.5cm}}
      $C_{BP}$ & $c_x$ & $\kappa_x$ &   $c_\alpha$ & $\kappa_\alpha$ 
      \\ \hline
      0.02  & 0.5275 & 0.3373 & 3.3750  & 3.0000  \\
      0.05  & 0.5241 & 0.3337 & 2.2241  & 2.8613 \\
      0.1  & 0.7188 & 0.3426 & 5.8352  & 2.5925  \\\hline
    \end{tabular}
    \vspace{5mm}
    \caption{Exponential growth model with
      $\xdag(t)\equiv 1; \; (0 < t \leq 1)$. Numerically computed
      results for the balancing principle~\eqref{eq:Hdelta-def}
      yielding multipliers and exponents of regularization
      error~\eqref{eq:regression_model_x} and
      $\alpha$-rates~\eqref{eq:regression_model_alpha} for various
      $C_{BP}$.}
    \label{tab:rates_BP}
  \end{center}
\end{table}

Figure~\ref{fig:BP_ratecomp}
\begin{figure}[h!]
  \begin{center}
    \subfigure{\includegraphics[width=0.48\textwidth]{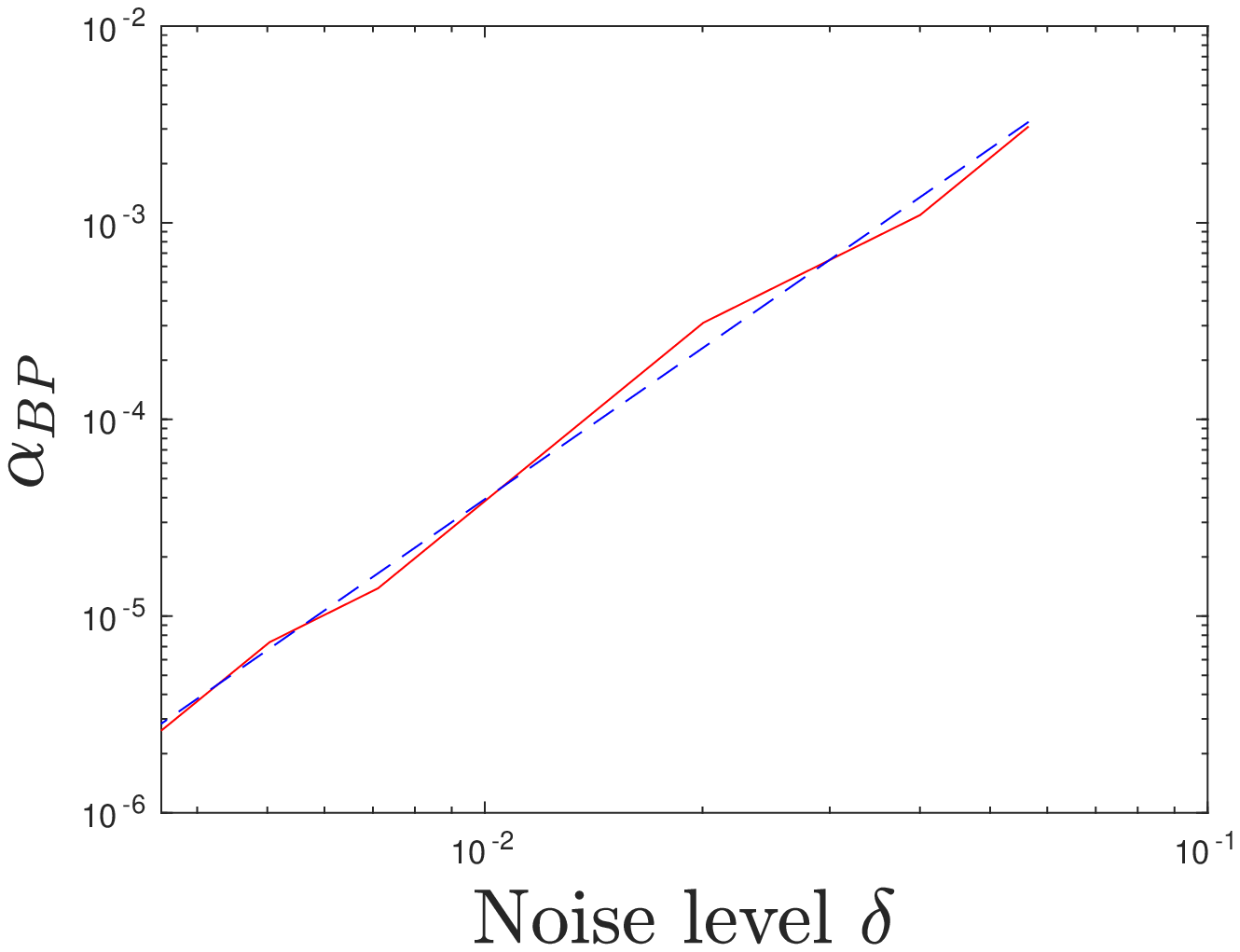}}
    \subfigure{\includegraphics[width=0.48\textwidth]{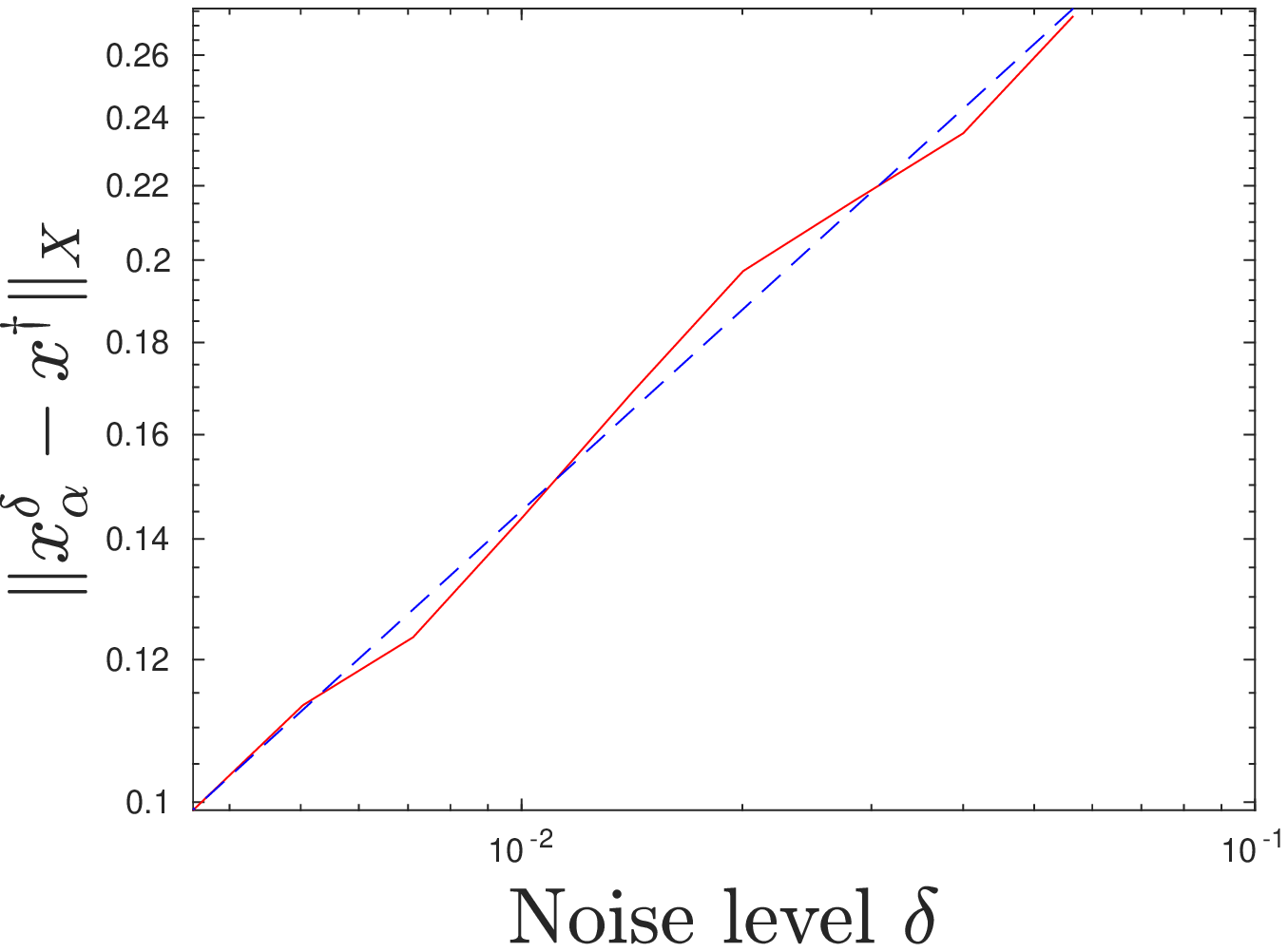}}
    \caption{Exponential growth model with
      $\xdag(t)\equiv 1; \; (0 < t \leq 1)$ and parameter choice using
      the balancing principle~\eqref{eq:Hdelta-def} with
      $C_{BP}=0.1$. $\alpha_{\scriptstyle BP}$ in red for various
      $\delta$ and best approximating regression line in blue/dashed
      on a log-log scale (left) and approximation error
      $\|\xad-\xdag\|_X$ in red and approximate rate in blue/dashed
      (right).}
    \label{fig:BP_ratecomp}
  \end{center}
\end{figure}
shows the realized regularization parameters
$\alpha_{\scriptstyle BP}$ (left) obtained by the balancing principle and the corresponding regularization errors (right) as well as their
respective approximations. We observe an excellent fit which confirms
our confidence in this approach, and in the implementation.

Next we fix the noise level~$\delta$ and compare various parameter
choice rules. Besides the balancing principle, we consider a
discrepancy principle, where the parameter is chosen a posteriori such
that
\begin{equation} \label{eq:intr_DP} \|F(\xad)-\yd\|_Y =
  C_{\scriptstyle DP}\, \delta
\end{equation}
for given noise level $\delta$ and suitable constant
$C_{\scriptstyle DP}>0$ . Recall, that this parameter choice rule also
yields the order optimal convergence rate.  Additionally we study the
heuristic parameter choice originally developed by Tikhonov, Glasko
and Leonov\footnote{The authors in the respective publications call
  this parameter choice quasioptimality. In order to avoid ambiguity
  with respect to Definition~\ref{th:quasioptimal} we avoid this
  terminology here, but denote the resulting regularization parameter
  as $\alpha_{\scriptstyle QO}$.}~\cite{Leonov91,TikGla64} violating
the Bakushinskij veto established in \cite{Baku84}.  \ Hence we
consider a sequence of regularization parameters
\begin{equation}\label{eq:qo_parseries}
  \alpha_k=\alpha_0 q^{k} \colon j=0,1,\ldots,M
\end{equation}
for $q>0$, some suitable $\alpha_0$ and appropriate $M$. Here,
$x_{\alpha_k}^{\delta}$ denotes the regularized solutions to the
functional \eqref{eq:tikhonov} with regularization parameter
${\alpha_k}$.  Using the series of parameters~\eqref{eq:qo_parseries}
the suitable regularization parameter according to this parameter
choice rule is then chosen by minimizing
\begin{equation*}
  \| x_{\alpha_{k+1}}^{\delta} - x_{\alpha_k}^{\delta} \|_X \rightarrow \min \quad 1\leq k \leq M-1
\end{equation*}
with respect to $\alpha_k$.  Moreover consider
$\alpha_{\scriptstyle \myopt}$ which minimizes the error
\mbox{$\| \xad-\xdag\|_X$}, i.e.,
\begin{equation}\label{eq:alpha_opt}
  \alpha_{\scriptstyle \myopt} := \min_{\alpha} \| \xad-\xdag\|_X
\end{equation}
assuming the exact solution is known.

All of the above parameter choices are visualized
\begin{figure}[h!]
  \begin{center}
    \subfigure{\includegraphics[width=1\textwidth]{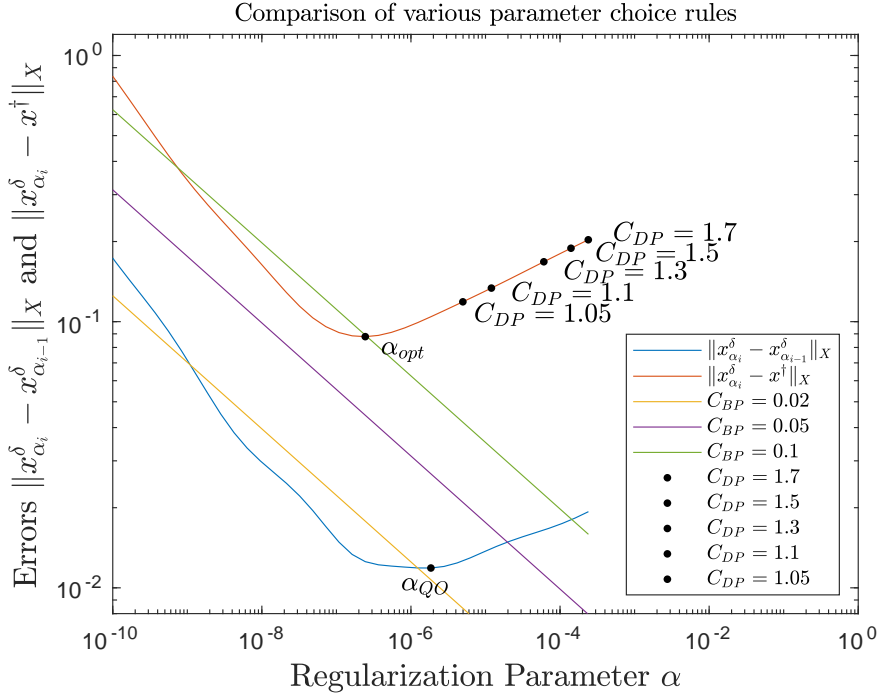}}
    \caption{Exponential growth model with
      $\xdag(t)\equiv 1; \; (0 < t \leq 1)$ and
      $\delta=0.0179$. Visualization of
      $\| x_{\alpha_{k+1}}^{\delta} - x_{\alpha_{k}}^{\delta} \|_X$
      and $\| \xad-\xdag\|_X$ as well as parameter choice using the
      balancing principle, discrepancy principle, quasi-optimality and
      $\alpha_{\scriptstyle \myopt}$ for various $C_{BP}$ and
      $C_{\scriptstyle DP}$.}
    \label{fig:BP_parcompDP}
  \end{center}
\end{figure}
for~$\delta=0.0179$  in
Figure~\ref{fig:BP_parcompDP}. Regularization error $\| \xad-\xdag\|_X$
and the term $\|x_{\alpha_i}^{\delta}-x_{\alpha_{i-1}}^{\delta}\|_X$
are plotted for various regularization parameters on a log-log
scale. Their respective minima are marked as parameter choices
$\alpha_{\scriptstyle \myopt}$ and $\alpha_{\scriptstyle QO}$. The
colored graphs correspond to the right hand term of the balancing
principle~(\ref{eq:bal-mod-bernd}) for different choices of $C_{BP}$. The
parameter choice rule can then be interpreted in the following way:
the balancing principle chooses the largest regularization parameter
from the admissible set, such that the left hand term
in~(\ref{eq:bal-mod-bernd}) is less or equal to the right hand
side. Visually speaking this means choosing the regularization
parameter at the intersect or just below the intersect of the blue and
colored lines, again depending on $C_{BP}$. Larger $C_{BP}$ leads to
larger regularization parameters and vice versa. The resulting
parameters for parameter choice using the discrepancy principle and
the respective regularization error are also marked for various
$C_{\scriptstyle DP}$.

In this situation, with fixed noise level, we observe that the
heuristic parameter choice $\alpha_{\scriptstyle QO}$ performs
surprisingly well. Parameter choice using the discrepancy principle
follows our intuition: for smaller constants~$C_{\scriptstyle DP}$ the
regularization error decreases and vice versa. The success of the
balancing principle highly depends on the choice $C_{BP}$. Although
theoretical results on the choice of this constant exist it is
difficult to chose this accordingly in practice.

Finally we highlight the differences between oversmoothing to
non-oversmoothing penalties. We therefore remain in the same setting
and consider regularized solutions
$\xdag(t)\equiv 1 \; (0 < t \leq 1)$ with
$\xdag \in X_p \; (0<p<\frac{1}{2})$ (oversmoothing case) and
\mbox{$\hat{x}^{\dag}(t)=-(t-\frac{1}{2})^2+\frac{1}{4} \; (0 < t
  \leq 1)$} with $\hat{x}^{\dag} \in X_1$ (non-oversmoothing case). We
again minimize the Tikho\-nov functional~\eqref{eq:tikhonov} for various
parameter choice rules. Understand that $\hat{x}^{\dag} \in X_p$ for
some $p>1$ and therefore the penalty is not oversmoothing. The exact
and regularized solution for various parameter choice are visualized
in Figure~\ref{fig:regsol_comp}. The left column considers $\xdag$,
the right column $\hat{x}^{\dag}$. Parameter choice
$\alpha \approx 9.52\mye{-09}$ is chosen a priori and too small in both
instances. Regularized solutions are displayed in the first row. We
therefore see highly oscillating regularized solutions and an
insufficient noise suppression. In the second row,
$\alpha \approx 2.44\mye{-07}$ is the optimal parameter choice in the sense
of~\eqref{eq:alpha_opt} for the oversmoothing situation. Similarly,
$\alpha \approx 2.12\mye{-05}$ (third row) is the optimal parameter choice
for the non oversmoothing situation. We see that the first parameter
choice leads to highly oscillating regularized solutions in the non
oversmoothing case.  Further, we obverse a phenomenon inherent to
regularization with oversmoothing regularization (left): when
comparing the parameter choices $\alpha \approx 2.44\mye{-07}$ and
$\alpha \approx 2.12\mye{-05}$ it becomes evident that the regularized
solution for the first parameter choice oscillates mildly, while the
latter appears much smoother. This occurs, as regularized solution
have to adhere to the boundary condition $\xdag(1)=0$. Therefore a
trade off between noise suppression and boundary condition
occurs. This is in agreement with results in~\cite{HofMat18,HofPla20}.
The fourth column shows the regularized solutions for too large
regularization parameters $\alpha \approx 1.60\mye{-04}$. Noise is
effectively suppressed, but in both instances the regularized
solutions are too smooth.

\begin{figure}[h!]
  \begin{tabular}{c c}
    \includegraphics[width=0.5\linewidth]{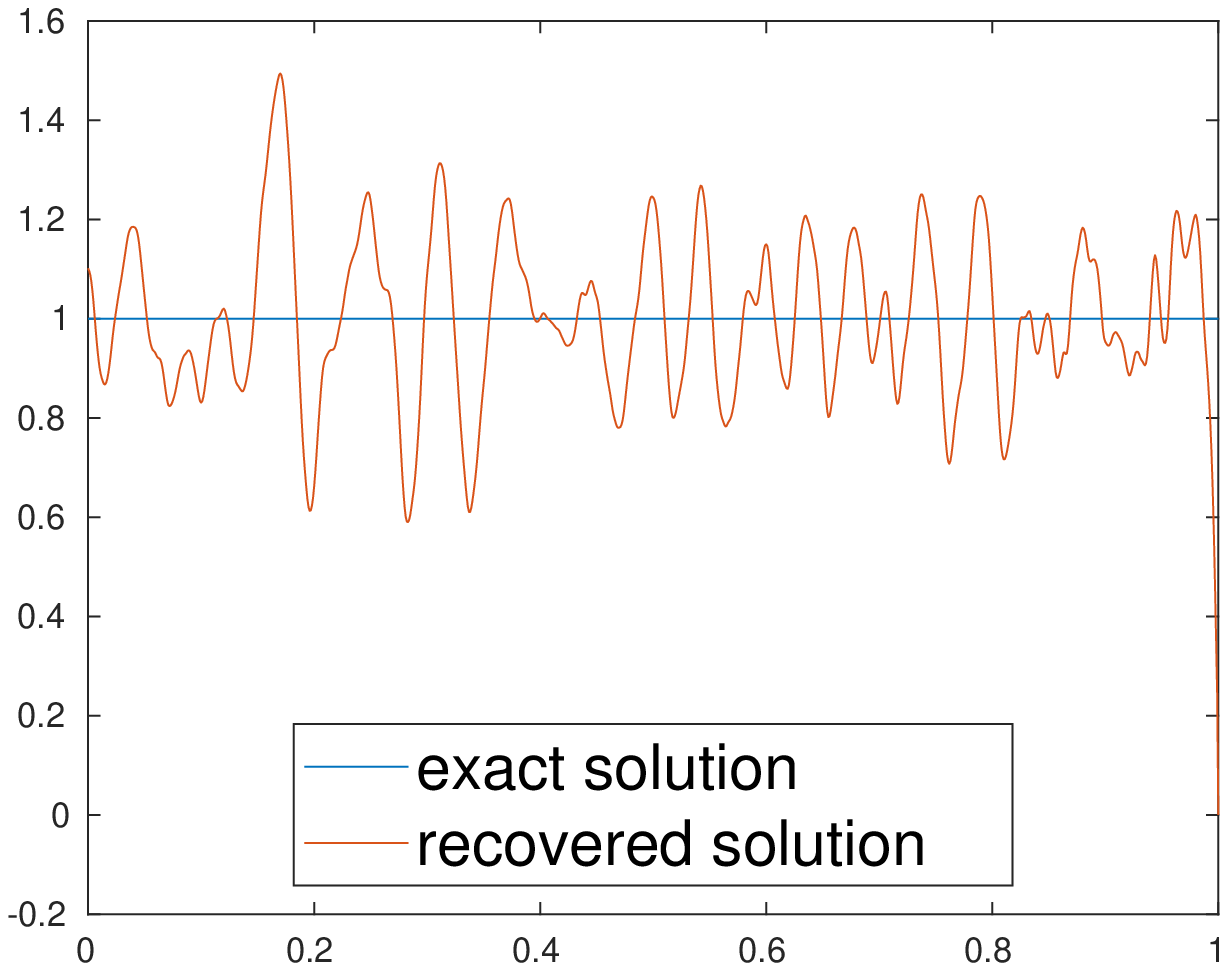} &\includegraphics[width=0.5\linewidth]{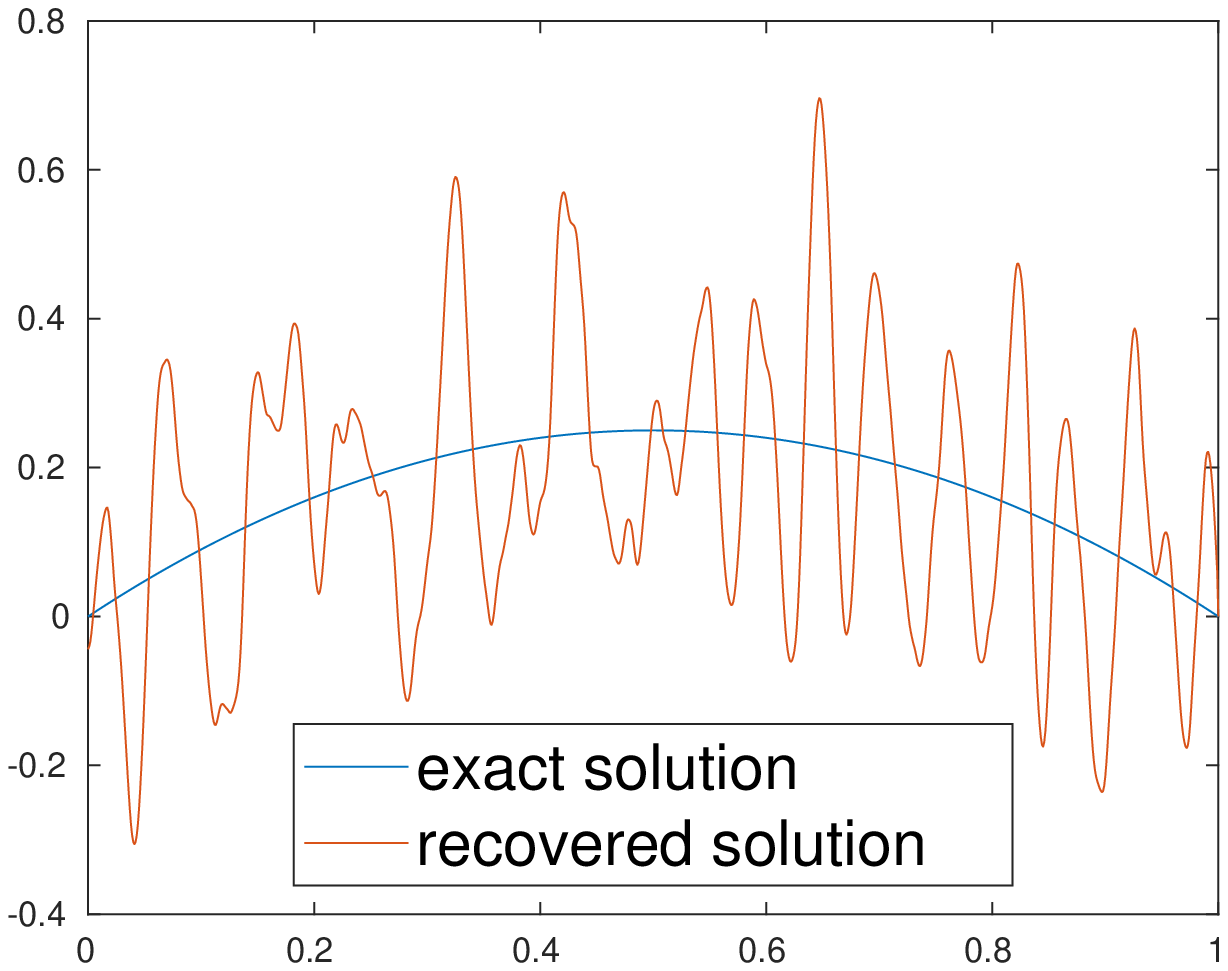}\\
    \multicolumn{2}{c}{$\alpha \approx 10^{-8}$} \\
    \includegraphics[width=0.5\linewidth]{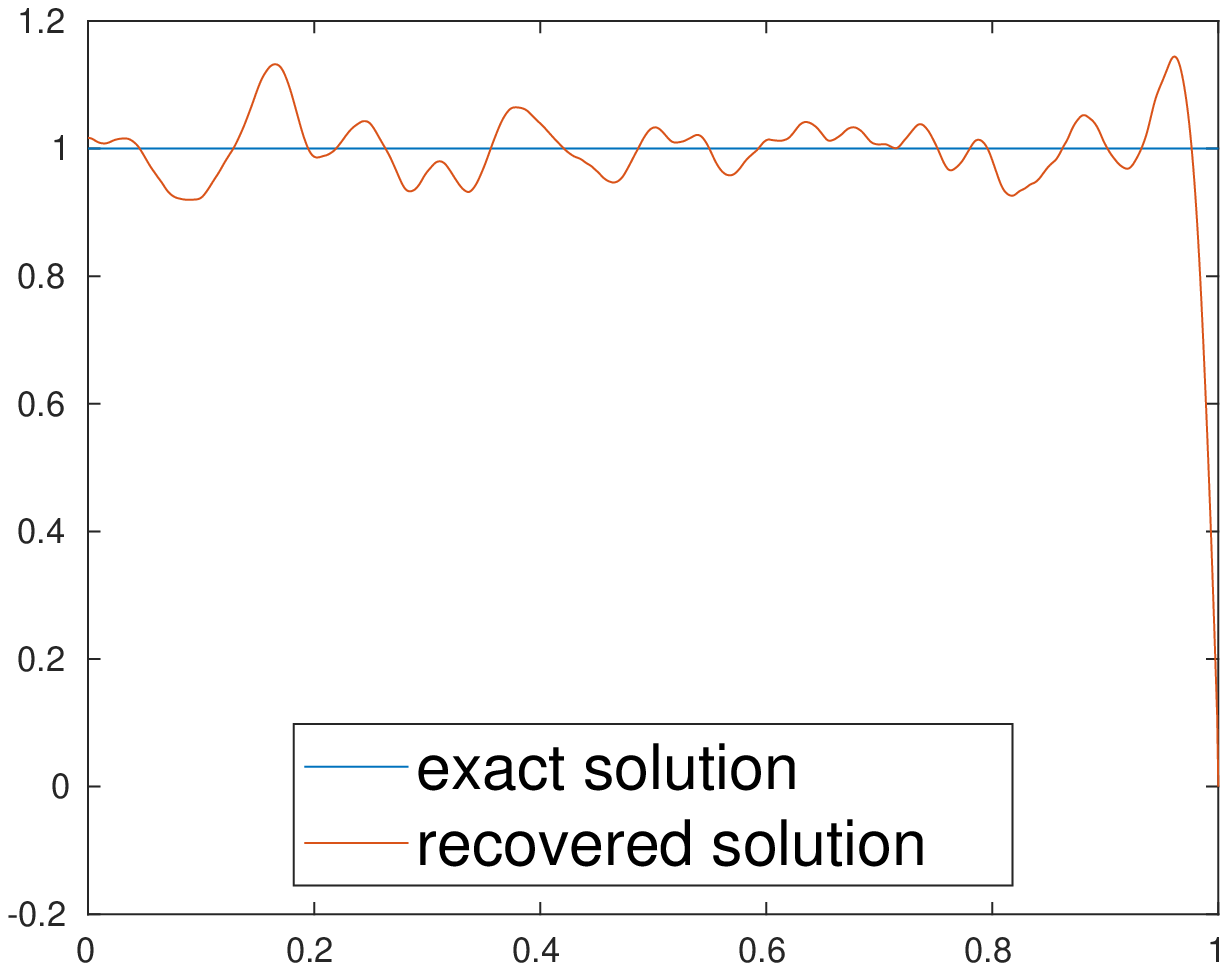} &\includegraphics[width=0.5\linewidth]{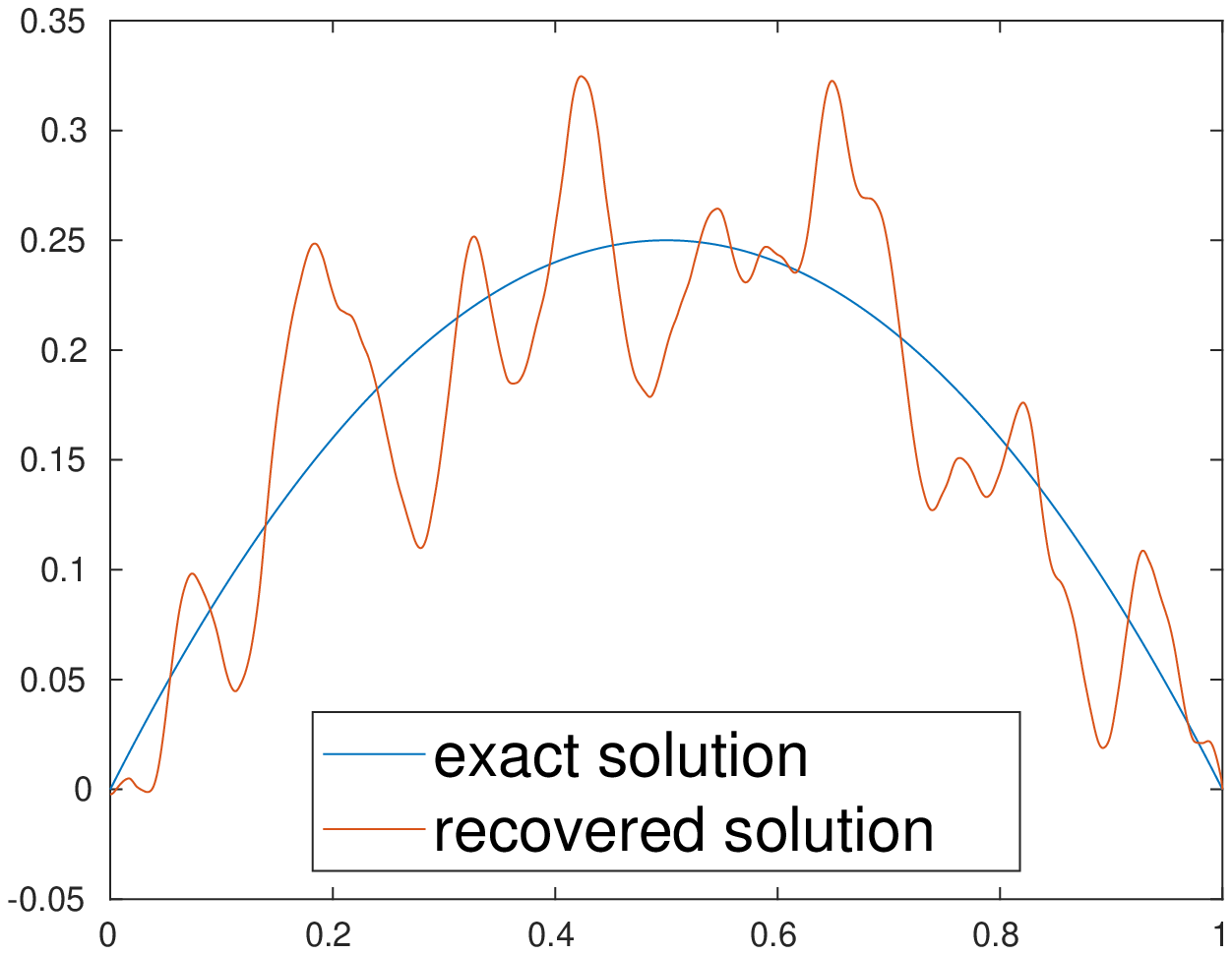}\\
    \multicolumn{2}{c}{$\alpha \approx 2.4\times 10^{-7}$} \\
    \includegraphics[width=0.5\linewidth]{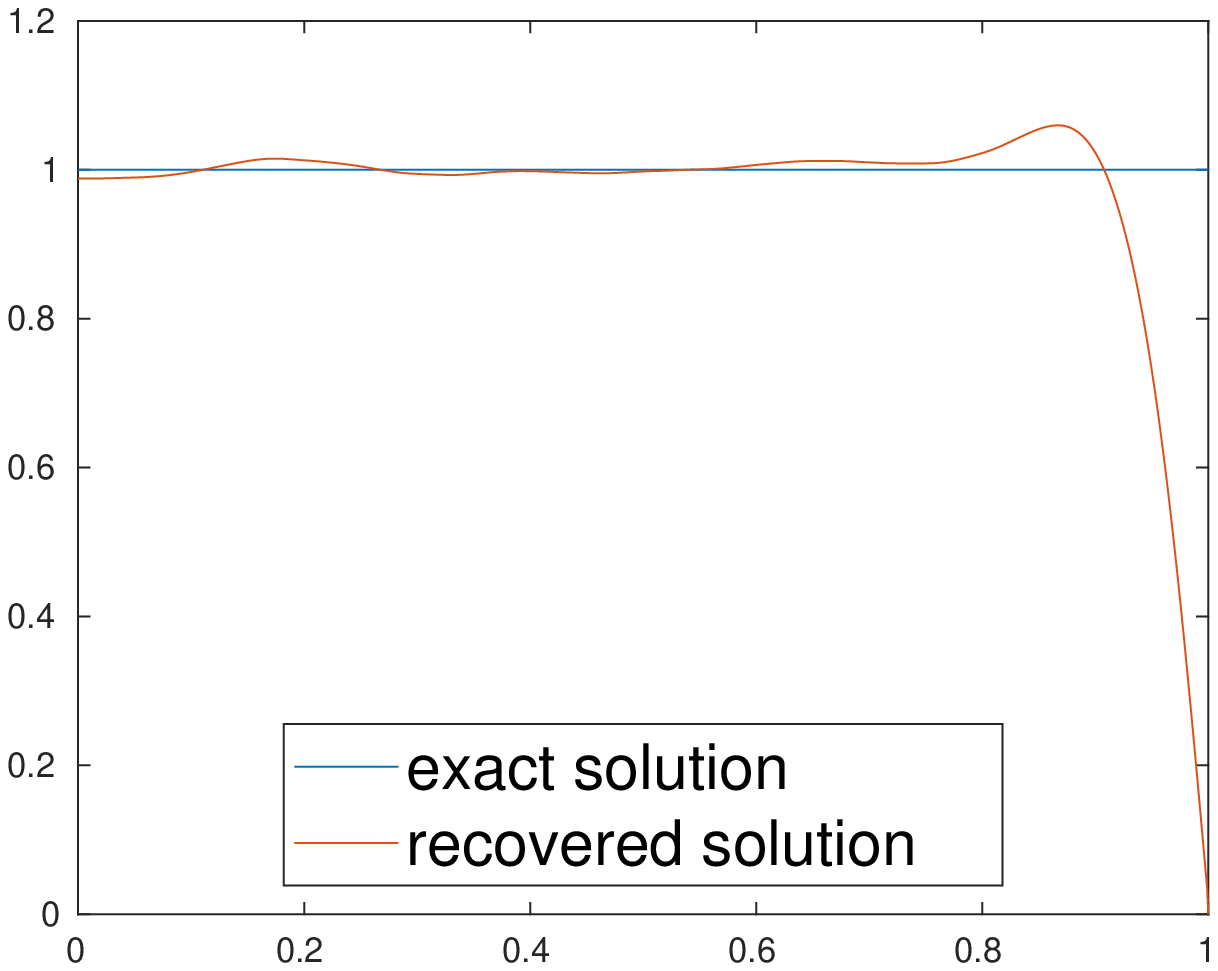} &\includegraphics[width=0.5\linewidth]{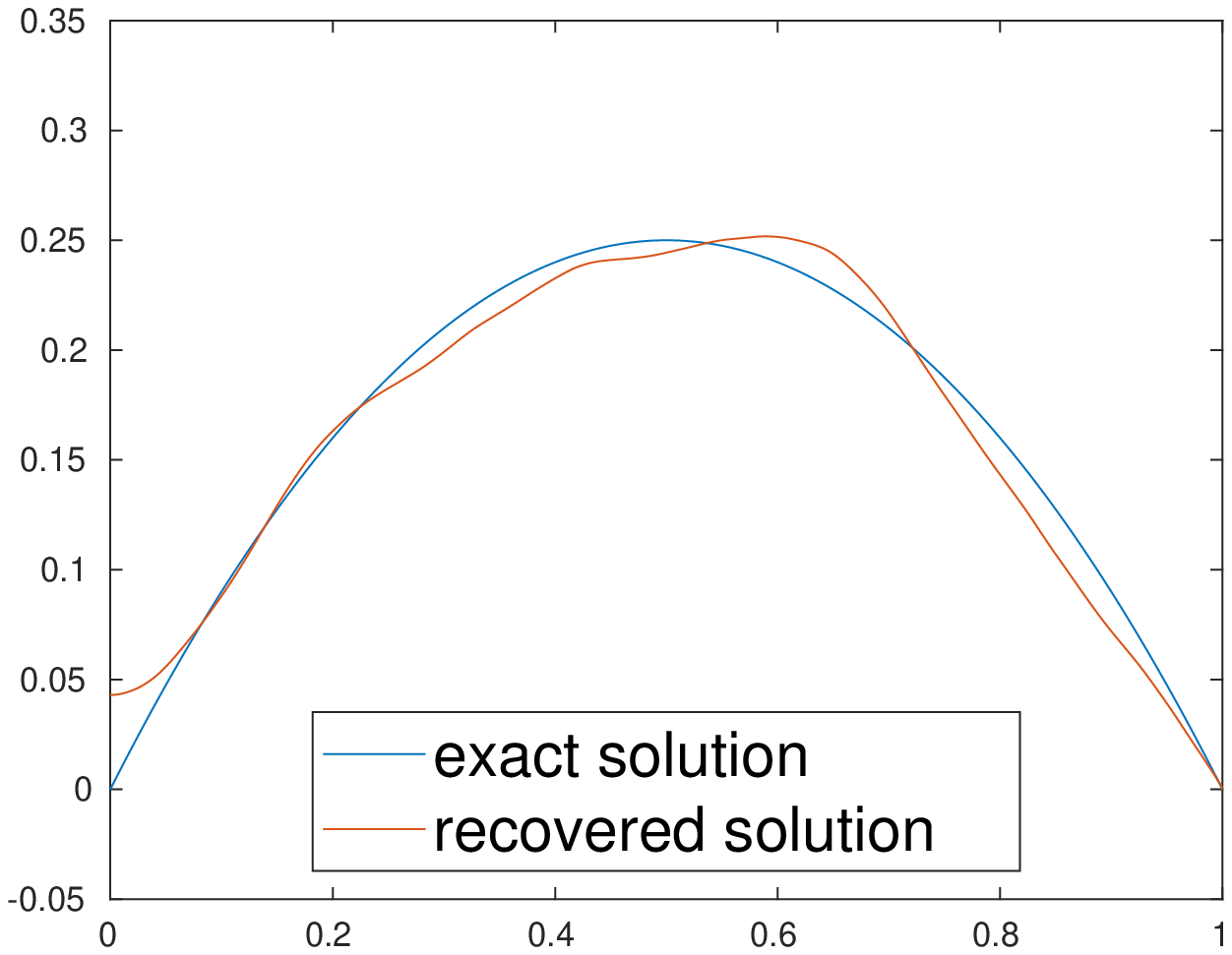}\\
    \multicolumn{2}{c}{$\alpha \approx 2.1\times 10^{-5}$} \\
    \includegraphics[width=0.5\linewidth]{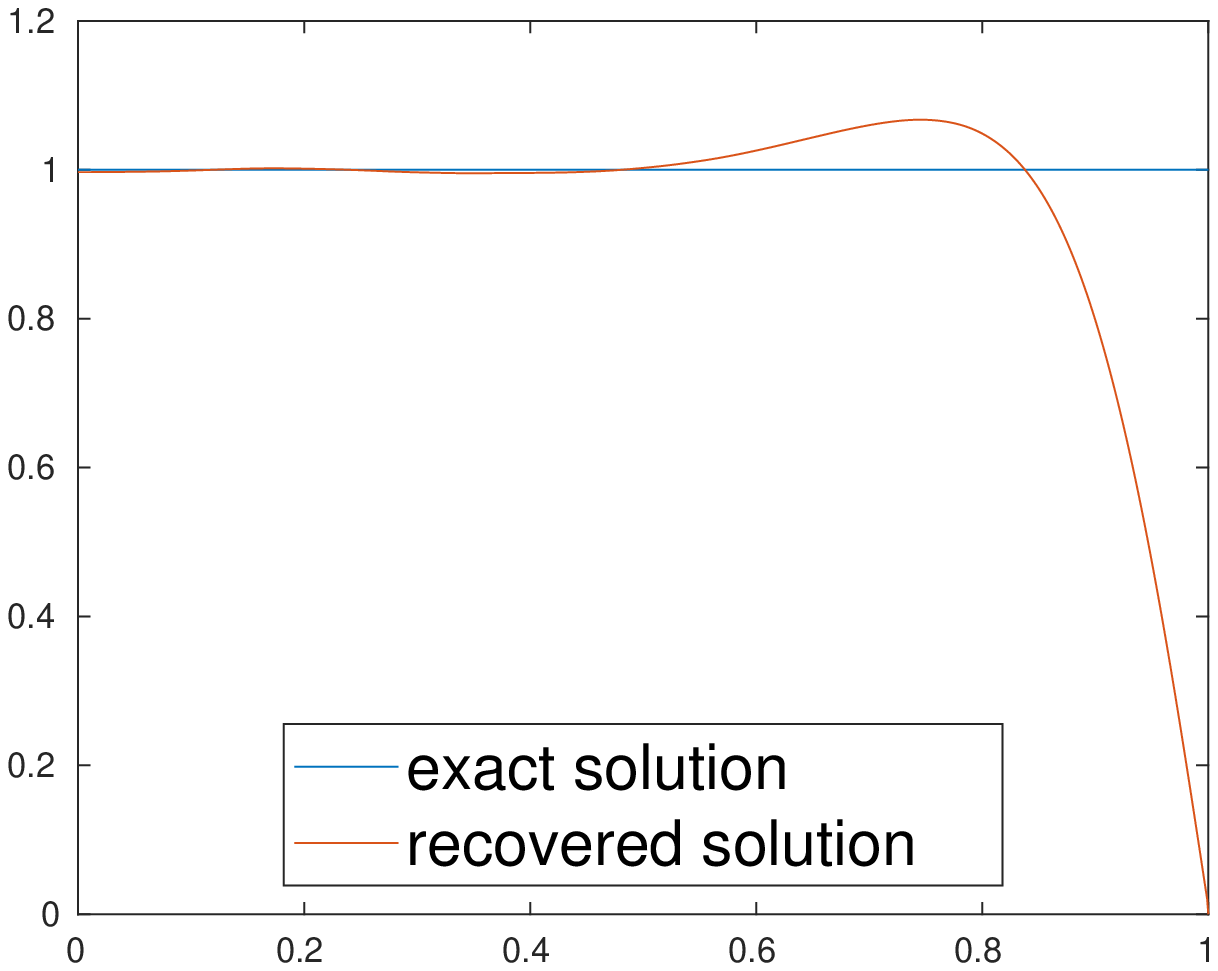} &\includegraphics[width=0.5\linewidth]{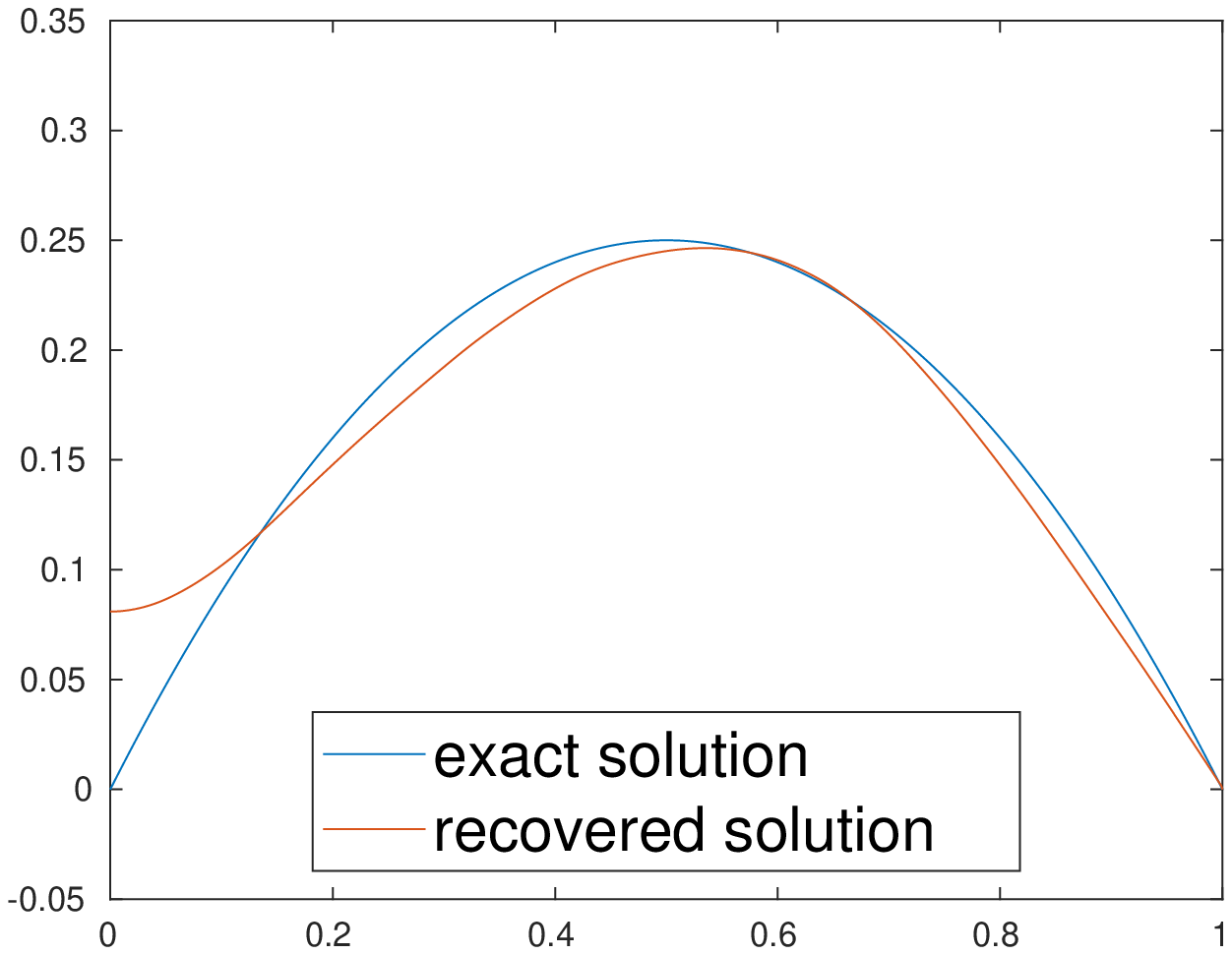}\\
    \multicolumn{2}{c}{$\alpha \approx 1.6\times 10^{-4}$} \\
  \end{tabular}\caption{Exponential growth model with $\xdag(t)\equiv 1; \; (0 < t \leq 1)$ (left) and $\hat{x}^{\dag}(t)=-(t-\frac{1}{2})^2+\frac{1}{4}; \; (0 < t \leq 1)$, $\delta=0.0179$. Regularized and exact solutions for various regularization parameters.}\label{fig:regsol_comp}
\end{figure}

\bibliographystyle{plain}
\bibliography{hhmp}

\begin{thebibliography}{10}

\bibitem{Baku84}
A.~B. Bakushinskij.
\newblock Remarks on choosing a regularization parameter using the
  quasi-optimality and ratio criterion.
\newblock {\em U.S.S.R.~Comput.~Math.~Math.~Phys.}, 24(4):181--182, 1984.
\newblock Translation from {\it Zh.~Vychisl.~Mat~Mat.~Fiz.}, 24(8), 1258--1259,
  1984 (Russian).

\bibitem{GHH20}
D.~Gerth, B.~Hofmann, and C.~Hofmann.
\newblock Case studies and a pitfall for nonlinear variational regularization
  under conditional stability.
\newblock In J.~Cheng, S.~Lu, and M.~Yamamoto, editors, {\em Inverse Problems
  and Related Topics: Shanghai, China, October 12--14, 2018}, Springer
  Proceedings in Mathematics \& Statistics, Vol.~310, Chapter~9, pages
  177--203. Springer Nature, Singapore, 2020.

\bibitem{GorYam99}
R.~Gorenflo and M.~Yamamoto.
\newblock Operator-theoretic treatment of linear {A}bel integral equations of
  first kind.
\newblock {\em Japan J.~Indust.~Appl.~Math.}, 16(1):137--161, 1999.

\bibitem{Groe93}
C.~W. Groetsch.
\newblock {\em Inverse Problems in the Mathematical Sciences}.
\newblock Vieweg Mathematics for Scientists and Engineers. Friedr. Vieweg \&
  Sohn, Braunschweig, 1993.

\bibitem{Haemarik_Raus[09]}
U.~H\"amarik and T.~Raus.
\newblock About the balancing principle for choice of the regularization
  parameter.
\newblock {\em Numer.~Funct.~Anal.~Opt.}, 30(9--10):951--970, 2009.

\bibitem{Hof98}
B.~Hofmann.
\newblock A local stability analysis of nonlinear inverse problems.
\newblock In D.~Delaunay et~al., editor, {\em Inverse Problems in Engineering -
  Theory and Practice}, pages 313--320. The American Society of Mechanical
  Engineers, New York, 1998.

\bibitem{HofHof20}
B.~Hofmann and C.~Hofmann.
\newblock The impact of the discrepancy principle on the {T}ikhonov-regularized
  solutions with oversmoothing penalties.
\newblock {\em mathematics - www.mdpi.com/journal/mathematics}, 8(3):331
  (16pp), 2020.

\bibitem{HKPS07}
B.~Hofmann, B.~Kaltenbacher, C.~P\"{o}schl, and O.~Scherzer.
\newblock A convergence rates result for {T}ikhonov regularization in {B}anach
  spaces with non-smooth operators.
\newblock {\em Inverse Problems}, 23(3):987--1010, 2007.

\bibitem{HofMat18}
B.~Hofmann and P.~Math\'{e}.
\newblock {T}ikhonov regularization with oversmoothing penalty for non-linear
  ill-posed problems in {H}ilbert scales.
\newblock {\em Inverse Problems}, 34(1):015007 (14pp), 2018.

\bibitem{HofMat20}
B.~Hofmann and P.~Math\'e.
\newblock A priori parameter choice in {T}ikhonov regularization with
  oversmoothing penalty for non-linear ill-posed problems.
\newblock In J.~Cheng, S.~Lu, and M.~Yamamoto, editors, {\em Inverse Problems
  and Related Topics: Shanghai, China, October 12--14, 2018}, Springer
  Proceedings in Mathematics \& Statistics, Vol.~310, Chapter~8, pages
  169--176. Springer Nature, Singapore, 2020.

\bibitem{HMvW2009}
B.~Hofmann, P.~Math\'{e}, and H.~von Weizs\"{a}cker.
\newblock Regularization in {H}ilbert space under unbounded operators and
  general source conditions.
\newblock {\em Inverse Problems}, 25(11):115013 (15pp), 2009.

\bibitem{HofPla18}
B.~Hofmann and R.~Plato.
\newblock On ill-posedness concepts, stable solvability and saturation.
\newblock {\em J.~Inverse Ill-Posed Probl.}, 26(2):287--297, 2018.

\bibitem{HofPla20}
B.~Hofmann and R.~Plato.
\newblock Convergence results and low order rates for nonlinear {T}ikhonov
  regularization with oversmoothing penalty term.
\newblock {\em Electronic Transactions on Numerical Analysis}, 53:313--328,
  2020.

\bibitem{Hohage00}
T.~Hohage.
\newblock Regularization of exponentially ill-posed problems.
\newblock {\em Numer. Funct. Anal. Optim.}, 21(3--4):439--464, 2000.

\bibitem{Leonov91}
A.~S. Leonov.
\newblock On the accuracy of {T}ikhonov regularizing algorithms and the
  quasi-optimal choice of regularization parameter.
\newblock {\em Dokl. Akad. Nauk SSSR}, 321(3):460--465, 1991.

\bibitem{MR1091202}
O.~V. Lepski\u{\i}.
\newblock A problem of adaptive estimation in {G}aussian white noise.
\newblock {\em Teor. Veroyatnost. i Primenen.}, 35(3):459--470, 1990.

\bibitem{Mathe06}
P.~Math\'{e}.
\newblock The {L}epski\u{\i} principle revisited.
\newblock {\em Inverse Problems}, 22(3):L11--L15, 2006.

\bibitem{MathHof2008}
P.~Math\'{e} and B.~Hofmann.
\newblock How general are general source conditions?
\newblock {\em Inverse Problems}, 24(1):015009 (5pp), 2008.

\bibitem{MatPer03}
P.~Math{\'e} and S.~V. Pereverzev.
\newblock Geometry of linear ill-posed problems in variable {H}ilbert scales.
\newblock {\em Inverse Problems}, 19(3):789--803, 2003.

\bibitem{Natterer84}
F.~Natterer.
\newblock Error bounds for {T}ikhonov regularization in {H}ilbert scales.
\newblock {\em Appl.~Anal.}, 18(1-2):29--37, 1984.

\bibitem{PerSchock05}
S.~Pereverzev and E.~Schock.
\newblock On the adaptive selection of the parameter in regularization of
  ill-posed problems.
\newblock {\em SIAM J. Numer. Anal.}, 43(5):2060--2076, 2005.

\bibitem{Plato[17]}
R.~Plato.
\newblock The regularizing properties of multistep methods for first kind
  {V}olterra integral equations with smooth kernels.
\newblock {\em Comput. Methods Appl. Math.}, 17(1):139--159, 2017.

\bibitem{Pricop19}
M.~Pricop-Jeckstadt.
\newblock Nonlinear {T}ikhonov regularization in {H}ilbert scales with
  balancing principle tuning parameter in statistical inverse problems.
\newblock {\em Inverse Probl.~Sci.~Eng.}, 27(2):205--236, 2019.

\bibitem{Raus_Haemarik[07]}
T.~Raus and U.~H\"amarik.
\newblock On the quasioptimal regularization parameter choices for solving
  ill-posed problems.
\newblock {\em J. Inv. Ill-Posed Problems}, 15:419--439, 2007.

\bibitem{Scherzetal09}
O.~Scherzer, M.~Grasmair, H.~Grossauer, M.~Haltmeier, and F.~Lenzen.
\newblock {\em {V}ariational {M}ethods in {I}maging}, volume 167 of {\em
  Applied Mathematical Sciences}.
\newblock Springer, New York, 2009.

\bibitem{SKHK12}
T.~Schuster, B.~Kaltenbacher, B.~Hofmann, and K.S. Kazimierski.
\newblock {\em Regularization Methods in Banach Spaces}.
\newblock Walter de Gruyter, Berlin/Boston, 2012.

\bibitem{TikGla64}
A.~N. Tikhonov and V.~B. Glasko.
\newblock The approximate solution of {F}redholm integral equations of the
  first kind.
\newblock {\em U.S.S.R.~Comput.~Math.~Math.~Phys.}, 4(3):236--247, 1964.
\newblock Translation from {\it Zh.~Vychisl.~Mat~Mat.~Fiz.}, 4(3), 564--571,
  1964 (Russian).

\bibitem{Zeidler90}
E.~Zeidler.
\newblock {\em Nonlinear Functional Analysis and Its Applications - Nonlinear
  Monotone Operators}, volume {II}/{B}.
\newblock Springer-Verlag, New York, 1990.
\newblock Translated from the German.

\end{thebibliography}

\end{document}